\documentclass[12pt,a4paper]{article}%
\usepackage{amsmath}
\usepackage{graphicx}
\usepackage{epsfig}%
\usepackage{amsfonts}%
\usepackage{amssymb}
\usepackage{algorithm}
\usepackage{algorithmic}
\usepackage{color}

\usepackage{tikz} \usetikzlibrary{3d,positioning,shapes}
\usetikzlibrary{scopes,shapes.geometric,shadows}
\usepackage{pgfplots} \usetikzlibrary{plotmarks}
\usetikzlibrary{calc,backgrounds,fit}

\newtheorem{theorem}{Theorem}[section]
\newtheorem{corollary}[theorem]{Corollary}

\newtheorem{definition}[theorem]{Definition}

\newtheorem{lemma}[theorem]{Lemma}
\newtheorem{proposition}[theorem]{Proposition}
\newtheorem{remark}[theorem]{Remark}

\newenvironment{proof}[1][Proof]{\noindent \emph{#1.} }{\hfill \ 
\rule{0.5em}{0.5em}}

\makeatletter\@addtoreset{equation}{section}\makeatother
\makeatletter\@addtoreset{figure}{section}\makeatother
\makeatletter\@addtoreset{table}{section}\makeatother
\begin{document}

\title{Block circulant and Toeplitz structures in the linearized  
Hartree-Fock equation on finite lattices: tensor approach}

\author{V. Khoromskaia \thanks{Max-Planck-Institute for
        Mathematics in the Sciences, Inselstr.~22-26, D-04103 Leipzig,
        Germany  ({\tt vekh@mis.mpg.de}); 
        Max Planck Institute for Dynamics of Complex Systems, Magdeburg, Germany.} \and
        B. N. Khoromskij \thanks{Max-Planck-Institute for
        Mathematics in the Sciences, Inselstr.~22-26, D-04103 Leipzig ({\tt bokh@mis.mpg.de}); 
        Max Planck Institute for Dynamics of Complex Systems, Magdeburg.}
        }
 
\date{}

\maketitle

\begin{abstract}
This paper introduces and analyses the new grid-based tensor approach
to approximate solution of the elliptic eigenvalue problem
for the 3D lattice-structured systems.
We consider the linearized Hartree-Fock equation over a spatial 
$L_1\times L_2\times L_3$ lattice for both periodic and non-periodic problem setting, 
discretized in the localized Gaussian-type orbitals basis.
In the periodic case, the Galerkin system matrix obeys a
three-level block-circulant structure that allows the FFT-based diagonalization, while
for the finite extended systems in a box (Dirichlet boundary conditions)
we arrive at the perturbed block-Toeplitz representation providing fast
matrix-vector multiplication and low storage size.
The proposed  grid-based tensor techniques manifest the twofold benefits:
(a) the entries of the Fock matrix are computed by  1D operations using
low-rank tensors represented on a 3D grid, (b) in the periodic case the low-rank tensor
structure in the diagonal blocks of the Fock matrix in the Fourier space
reduces the conventional 3D FFT to the product of 1D FFTs.
Lattice type systems in a box with Dirichlet boundary conditions
are treated numerically by our previous tensor solver for single
molecules, which makes possible calculations on rather 
large $L_1\times L_2\times L_3$ lattices due to reduced numerical cost for 3D problems.
The numerical simulations for  both box-type and periodic $L\times 1\times 1$
lattice chain in a 3D rectangular ``tube'' with $L$ up to several hundred
confirm the theoretical complexity bounds for the block-structured 
eigenvalue solvers in the limit of large $L$.
\end{abstract}



\section{Introduction}\label{sec:introduct}

 Efficient numerical simulation of lattice systems for both periodic and non-periodic
settings in application to crystalline, metallic, polymer-type compounds, and nano-structures
is one of the challenging tasks in computational quantum chemistry.  
 The reformulation of the nonlinear Hartree-Fock equation for periodic
molecular systems based on the Bloch theory \cite{Bloch:1925}
has been addressed in the literature for more than forty years ago,
and nowadays there are several implementations mostly relying on the analytic 
treatment of arising integral operators \cite{CRYSTAL:2000,CRYSCOR:12,GAUSS:09}.
Mathematical analysis of spectral problems for PDEs with the
periodic-type coefficients was an attractive topic in the recent decade, see
\cite{CancesDeLe:08,OrtnEhrl:13,LuOrtnerVK:2013} and the references therein. 
However, the systematic  optimization of the 
basic numerical algorithms in the {\it ab initio} Hartree-Fock calculations for 
large lattice structured compounds with perturbed periodicity are largely unexplored.

The real space methods for single molecules based on 
the locally adaptive grids and multiresolution techniques have been discussed
in \cite{HaFaYaBeyl:04,SaadRev:10,Frediani:13,CanEhrMad:2012}. 
The grid-based tensor-structured approach
for solving the Hartree-Fock nonlinear spectral problem approximated
in the basis set of localized Gaussian-type-orbitals (GTO) has been developed
and proved to be efficient for moderate size molecular systems
 \cite{VKH_solver:13,KhorSurv:10,VeBoKh:Ewald:14}.

This paper presents the grid-based tensor approach  to the solution of elliptic eigenvalue problem 
for the 3D lattice-structured systems in a bounding box. 
We focus on the basic application to the linearized Hartree-Fock equation for 
extended systems composed of atoms or molecules located at nodes
of a $L_1\times L_2\times L_3$ finite lattice, 
for both open boundary conditions and periodic supercell.
The latter is useful because the structure of the respective Galerkin matrix 
(i.e., the Fock matrix) in the presence of defects can be treated as a small (local) 
perturbation to an ideally periodic system.
We consider the 3D model eigenvalue problem for the Fock operator confined to the core
Hamiltonian part, composed of the 3D Laplacian and the nuclear potential operator
describing the Coulomb interaction of electrons and nuclei, 
which requires a sum of the total electrostatic potential of nuclei in the considered
extended system. This is the typical example of a nontrivial elliptic eigenvalue problem
arising in the numerical modeling of electronic structure in large almost periodic 
molecular systems. We observed in numerical experiments that there is an irreducible 
difference  between the spectral data for periodic and non-periodic settings.

Computation of 3D lattice sums of a large number of long-distance Coulomb interaction 
potentials is one of the severe difficulties in the Hartree-Fock calculations for 
lattice-structured periodic or box-restricted systems. 
Traditionally this problem was treated by the so-called Ewald-type summation 
techniques \cite{Ewald:27,DYP:93} combined with the fast multipole expansion or/and FFT methods 
\cite{RochGreen:87,SundLos:17}, which
scale as $O(L^3 \log L)$, $L=\max\{L_1,L_2,L_3\}$, for both periodic and box-type lattice sums.
In this paper we apply the new, recently introduced method for summation of long-range 
potentials on lattices \cite{VeBoKh:Ewald:14,VeKhorEwTuck_NLLA:15}  
by using the assembled rank-structured
tensor decomposition. This approach reduces the cost of  summation for 
$L \times L \times L$ lattices to linear scaling in $L$, i.e. $O(L)$.

In the presented approach the Fock matrix is calculated directly by grid-based tensor 
numerical operations in the basis set of localized 
Gaussian-type-orbitals\footnote{GTO basis can be viewed as the special reduced basis 
constructed on the base of physical insight.} (GTO) first
specified by $m_0$ elements in the unit cell and then finitely replicated on 3D  
extended lattice structure in a box \cite{VKH_solver:13,VeKhorTromsoe:15}.
For numerical integration by using low-rank tensor formats all basis
functions are represented on the fine rectangular grid covering the whole computational box,
where we introduce either the Dirichlet or periodic boundary conditions. 

We show that in the case of finite lattices in a box 
(the Dirichlet boundary conditions) the core Hamiltonian exhibits
the $C(2d+1)$-diagonal block sparsity, see Lemma 2.1. In particular both the discrete 
Laplacian and the mass matrix reveal the block-Toeplitz structure. 
The nuclear potential operator can be constructed, in general, as for the large single molecule
in a box or by replication from the central unit cell to the whole lattice.
In the latter case we arrive at the block-Toeplitz structure which allows
the fast FFT-based matrix vector product.

For periodic boundary conditions (periodic supercell)
we do not impose explicitly the periodicity-like features of the 
solution by means of the approximation ansatz that is the common approach in the Bloch formalism.
Instead, the periodic properties of the considered system appear implicitly through the 
Toeplitz or circulant block structures in the Fock matrix. 
In case of periodic supercell the Fock matrix is proved to inherit 
the $d$-level symmetric block circulant form, that allows its  
diagonalization in the Fourier basis \cite{KaiSay_book:99,Davis} at the expense
$O(m_0^2 L^d \log L)$, $d=1,2,3$, see Lemma 3.4.
In the case of $d$-dimensional lattice, the weak overlap between lattice 
translated basis functions leads to banded block sparsity thus reducing the storage 
cost. 
Furthermore, we introduce the low-rank tensor structure to the diagonal blocks of 
the Fock matrix represented in the Fourier space which allows to reduce the numerical
cost to handle the block-circulant Galerkin matrix to linear 
scaling in $L$, $O(m_0^2 L \log L)$, see Theorem 3.3.



The presented numerical scheme can be further investigated in  
the framework of the reduced Hartree-Fock model \cite{CancesDeLe:08},
where the similar block-structure in the Coulomb term of the Fock matrix can be observed.
The Wannier-type basis functions constructed by 
the lattice translation of the localized molecular orbitals precomputed 
on the reference unit cell, can be also adapted to this algebraic framework.

The arising block-structured matrix representing the discretized 
core Hamiltonian, as well as some auxiliary function-related tensors arising, 
can be considered for further optimization by imposing the low-rank tensor formats, 
and in particular, 
the quantics-TT (QTT) tensor approximation \cite{KhQuant:09} of long vectors and large
matrices, 
which especially benefits in the limiting case  of large $L\times L \times L$
perturbed periodic systems. 
In the QTT approach the algebraic operations on the 
3D $n\times n\times n$ representation Cartesian grid can be implemented with 
logarithmic cost $O(\log n)$.
Literature surveys on tensor algebra and rank-structured tensor methods for multi-dimensional
PDEs can be found in \cite{Kolda,KhorSurv:10,VeKhorTromsoe:15,CichOseletc:2016}, 
see also 
\cite{HaKhSaTy:08,GavlKh_Caley:11,LinLinetc_SelInv:11,DoKhSavOs_mEIG:13,BeKhKh_BSE:15,BeDoKh2_BSE2:16} 
and \cite{VeKhorTromsoe:15,RahOsel:16,RahOsel:16_2}
concerning the low-rank decompositions in eigenvalue and electronic structure calculations.
The present paper represents the revised and essentially extended version of the previous preprint 
\cite{VeKhorCorePeriod:14}.

Notice that in the recent years the analysis of eigenvalue problem solvers for  
large structured matrices has been widely discussed in the linear algebra community 
\cite{BuByMe:92,BeMeXu:97,BuFa:15,BeFaYa:15}.
Tensor structured approximation of elliptic equations with quasi--periodic coefficients
has been considered in \cite{BokhSRep:15,BokhSRep2:16}.

The rest of the paper is organized as follows. 
Section \ref{sec:core_H} includes the main results on the analysis of
core Hamiltonian on lattice structured compounds. 
In particular, \S\ref{Core_Hamil} describes the tensor-structured calculation of the 
core Hamiltonian for large lattice-type molecular/atomic systems.
We recall tensor-structured calculation of the Laplace operator and
fast summation of lattice potentials by assembled canonical tensors.
The complexity reduction due to low-rank tensor structures in the
matrix blocks is discussed, see Remark \ref{prop:low_rank_coef}). 
Section \ref{sec:Core_Ham_period_FFT} discusses in detail the block circulant 
structure of the core Hamiltonian and presents numerical illustrations for 
a rectangular 3D ``tube'' of size $L\times 1\times 1$  with $L$ for large $L$.
In particular, \S\ref{ssec:Tensor_bcirc} introduces the new
block structures by imposing the low-rank factorizations within multi-indexed 
blocks of the diagonalized three-level block-circulant matrix. 
We present a number of numerical experiments illustrating the pollution effect on the spectrum
of periodic system compared with the system in a finite box. We also demonstrate 
the optimal performance for the direct FFT-based solver that implements the one-level
block-circulant matrix structure describing the $L \times 1 \times 1$ 
lattice systems for large $L$ (polymer-type compounds).
Appendix recalls the classical results on the properties of block circulant/Toeplitz matrices
and describes the basic tensor formats.

\section{Elliptic operators with lattice-structured potentials}
\label{sec:core_H}

In this section we  analyze the matrix structure of 
the Galerkin discretization for the elliptic eigenvalue problem in the form
\begin{equation}\label{eqn:EIGHFcore}
 {\cal H}\varphi({ x}) \equiv [- \Delta + v(x)]\varphi({ x}) =\lambda \varphi({ x}), \quad 
 x\in \Omega \subset \mathbb{R}^d, \;d=1,2,3,
\end{equation}
where the potential $v(x)$ is constructed
by replication of those in the rectangular unit cell $\Omega_0$ over a $d$-dimensional 
rectangular $L_1\times L_2\times L_3$ lattice in a box, such that  
$\varphi \in H^1_0(\Omega)$, 
or in a rectangular supercell $\Omega$ with periodic boundary conditions.
We focus on the important particular case of $v(x)= v_c(x)$ corresponding to
the core Hamiltonian part in the Fock operator that constitutes the
Hartree-Fock spectral problem arising in electronic structure calculations.
In this case the electrostatic potential $v_c(x)$ is obtained as the large
lattice sum of long-range interactions defined by the Newton kernel. 

\subsection{The Hartree-Fock core Hamiltonian in a GTO basis set}
\label{Core_Hamil}

The nonlinear Fock operator ${\cal F}$
in the governing Hartree-Fock eigenvalue problem,   
describing the ground state energy for $2N_b$-electron system, is defined by
\[
\left[-\frac{1}{2} \Delta - v_c(x) + 
 \int_{\mathbb{R}^3} \frac{\rho({ y})}{\|{ x}-{ y}\|}\, d{ y}\right] \varphi_i({ x})
- \int_{\mathbb{R}^3} \; \frac{\tau({ x}, { y})}{\|{x} - { y}\|}\, \varphi_i({ y}) d{ y}
 = \lambda_i \, \varphi_i({ x}), 
 \quad x\in \mathbb{R}^3,
\] 
where $i =1,...,N_{orb}$ and $\|\cdot \|$ means the distance function in $\mathbb{R}^3$.
The linear part in the Fock operator is presented by the core Hamiltonian 
\begin{equation}\label{eqn:HFcore}
{\cal H}=-\frac{1}{2} \Delta - v_c,
\end{equation}
while the nonlinear Hartree potential and exchange operators 
depend on the unknown eigenfunctions (molecular orbitals)
comprising the electron density, $\rho({ y})= 2 \tau(y,y)$, and the density matrix, 
$
\tau(x,y) =\sum\limits^{N_{orb}}_{i=1} \varphi_i (x)\varphi_i (y),\quad x,y\in \mathbb{R}^3.
$ 
The electrostatic potential generated by the core Hamiltonian is defined by a sum
\begin{equation} \label{eqn:ElectrostPot}
v_c(x)= \sum_{\nu=1}^{M}\frac{Z_\nu}{\|{x} -a_\nu \|},\quad
Z_\nu >0, \;\; x,a_\nu\in \mathbb{R}^3,
\end{equation}
where $M$ is the total number of nuclei in the system, 
$a_\nu$,  $Z_\nu$, represent their Cartesian coordinates and the respective charge numbers.


Given a set of localized GTO basis functions $\{g_\mu\}$ ($\mu=1,...,N_b$),
the occupied molecular orbitals $\psi_i$ are approximated in the form
 \begin{equation}\label{expand}
\psi_i=\sum\limits_{\mu=1}^{N_b} C_{\mu i} g_\mu, \quad i=1,...,N_{orb},
\end{equation}
with the unknown coefficients matrix
$C=[C_{ \mu i}] \in \mathbb{R}^{N_b \times  N_{orb}}$ obtained as the solution 
of the discretized Hartree-Fock equation with respect to the Galerkin basis $\{g_\mu\}$, 
and governed by $N_b\times N_b$ Fock matrix.
The stiffness matrix $H=[ h_{\mu \nu}]\in \mathbb{R}^{N_b\times N_b}$ of the core Hamiltonian
(\ref{eqn:HFcore}) is represented by the single-electron integrals, 
 \begin{equation}\label{eqn:Core_Ham}
h_{\mu \nu}= \frac{1}{2} \int_{\mathbb{R}^3}\nabla g_\mu \cdot \nabla g_\nu dx -
\int_{\mathbb{R}^3} v_c(x) g_\mu g_\nu dx, \quad 1\leq \mu, \nu \leq N_b,
\end{equation}
such that the resulting eigenvalue equations
governed by  the reduced Fock matrix, $H$, read 
\begin{align*} \label{eqn:HF discr}
  H C &= SC \Lambda, \quad \Lambda= diag(\lambda_1,...,\lambda_{N_{orb}}), \\
  C^T SC   &=  I_N,   \nonumber
\end{align*}
where the mass (overlap) matrix $S=[s_{\mu \nu}]_{1\leq \mu, \nu \leq N_b}$, is given by
$
s_{\mu \nu}=\int_{\mathbb{R}^3} g_\mu g_\nu  dx.
$

In the case of $L\times L\times L$ lattice system in a box 
the number of basis functions scales cubically in $L$, $N_b = m_0 L^3$, hence
the evaluation of the Fock matrix and further computations may become 
prohibitive as $L$ increases ($m_0$ is the number of basis functions in the unit cell).
Moreover, the numerically extensive part in the matrix evaluation (\ref{eqn:Core_Ham}) 
is related to the integration with the large sum of lattice translated Newton kernels. 
Indeed, let $M_0$ be the number of nuclei in the unit cell, then
the expensive calculations are due to the summation over $M_0 L^3$ Newton kernels,  
and further spacial integration of this sum with the large set of localized atomic orbitals 
$\{g_\mu\}$, ($\mu=1,...,N_b$), where $N_b$ is of order $m_0 L^3$.

In what follows, we describe the grid-based tensor approach for the  
block-structured representation
of the core Hamiltonian in the Fock matrix for the lattice system in a box or 
in a periodic supercell. 
The main ingredients of the present approach include:

(a) the fast and accurate grid-based tensor method for evaluation 
of the electrostatic potential $v_c$ defined by  the lattice sum 
in (\ref{eqn:ElectrostPot}),  see \cite{VeBoKh:Ewald:14,VeKhorEwTuck_NLLA:15};

(b) fast computation of the entries in the stiffness matrix $V_c$,
 \begin{equation}\label{eqn:Core_Ham_V}
V_c=[v_{\mu \nu}]:\quad  v_{\mu \nu}= \int_{\mathbb{R}^3} v_c(x) g_\mu g_\nu dx, 
\quad 1\leq \mu, \nu \leq N_b,
\end{equation}
by numerical grid-based integration using the low-rank tensor representation 
of all functions involved,  

(c) block-structured factorized representation of the large and densely populated 
matrix $V_c$ in the form of perturbed multilevel block-circulant matrix, and

(d) block representation of the Galerkin matrix for the Laplacian.


The approach provides the opportunities to reduce computational costs in the case of
large $L\times L \times L$ lattice systems.
In the next sections, we show that in the periodic setting the resultant stiffness matrix 
$H=[h_{\mu \nu}]$ of the core Hamiltonian 
can be parametrized in the form of a symmetric, three-level block circulant matrix
that allows further structural improvements by introducing tensor factorizations of
the matrix blocks.
In the case of lattice system in a box the block structure of $H$ is obtained by 
a small perturbation of the block Toeplitz matrix.
These matrix structures allow the efficient storage and fast matrix-vector multiplication
within iterations on a subspace for solving partial eigenvalue problem.

\subsection{Nuclear potential operator for a single molecule} 
\label{ssec:nuclear}

In this paragraph, we describe the evaluation of the stiffness matrix $V_c$ by tensor operations.
It is based on the low-rank separable approximation to the 
nuclear (core) potential $v_c(x)$ representing the Coulomb interaction 
of the electrons with the nuclei, see (\ref{eqn:ElectrostPot}). 

In the case $L=1$ we have the single (discrete) molecule embedded into the
scaled {\it unit cell}  $\Omega =[-\frac{b}{2},\frac{b}{2}]^3$. 
In the computational domain $\Omega$, we introduce the uniform $n \times n \times n$ rectangular 
Cartesian grid $\Omega_{n}$ with the mesh size $h=b/n$, and define 
the set of tensor-product piecewise constant finite element basis functions $\{ \psi_\textbf{i}\}$,
which are supposed to be separable, i.e.,
$  \psi_\textbf{i}(\textbf{x})=\prod_{\ell=1}^d \psi_{i_\ell}^{(\ell)}(x_\ell)$
for ${\bf i}=(i_1,i_2,i_3)$, $i_\ell \in I=\{1,...,n\}$.

Following \cite{BeHaKh:08,VKH_solver:13},
the Newton kernel is discretized by the projection/collocation method in the form
of a third order tensor  $\mathbb{R}^{n\times n \times n}$, defined by
\begin{eqnarray}
\mathbf{P}:=[p_{\bf i}] \in \mathbb{R}^{n\times n \times n},  \quad
 p_{\bf i} = 
\int_{\mathbb{R}^3} \frac{\psi_{{\bf i}}({x})}{\|{x}\|} \,\, \mathrm{d}{x}.
  \label{galten}
\end{eqnarray}
The low-rank canonical decomposition of the $3$rd order tensor $\mathbf{P}$ is based 
on using exponentially convergent 
$\operatorname*{sinc}$-quadratures approximation 
of the Laplace-Gauss transform, \cite{Braess:95,Stenger,GHK:05,HaKhtens:04I}, 
\[
 \frac{1}{z}= \frac{2}{\sqrt{\pi}}\int_{\mathbb{R}_+} e^{- z^2 t^2 } dt,\quad z>0,
\]
which can be adapted to the Newton kernel by substitution $z=\sqrt{x_1^2 + x_2^2  + x_3^2}$.
We denote the resulting $R$-term  canonical representation by
\begin{equation} \label{eqn:sinc_general}
    \mathbf{P} \approx  \mathbf{P}_R 
= \sum\limits_{q=1}^{R} {\bf p}^{(1)}_q \otimes {\bf p}^{(2)}_q \otimes {\bf p}^{(3)}_q
\in \mathbb{R}^{n\times n \times n}.
\end{equation}
We further suppose that all atomic centers are located strictly within subdomain 
$\Omega_0 =[-\frac{b_0}{2},\frac{b_0}{2}]^3\subset \Omega$, $b_0< b$,
called \emph{formation domain}, and
define the auxiliary (bounding) box $\widetilde{\Omega} \supset \Omega$,
associated with the grid parameter $\widetilde{n}=n_0+n$ (say, $\widetilde{n}=2 n$),
see Figure \ref{fig:Unit_cell}.
\begin{figure}[htbp]
\centering
\includegraphics[width=3.0cm]{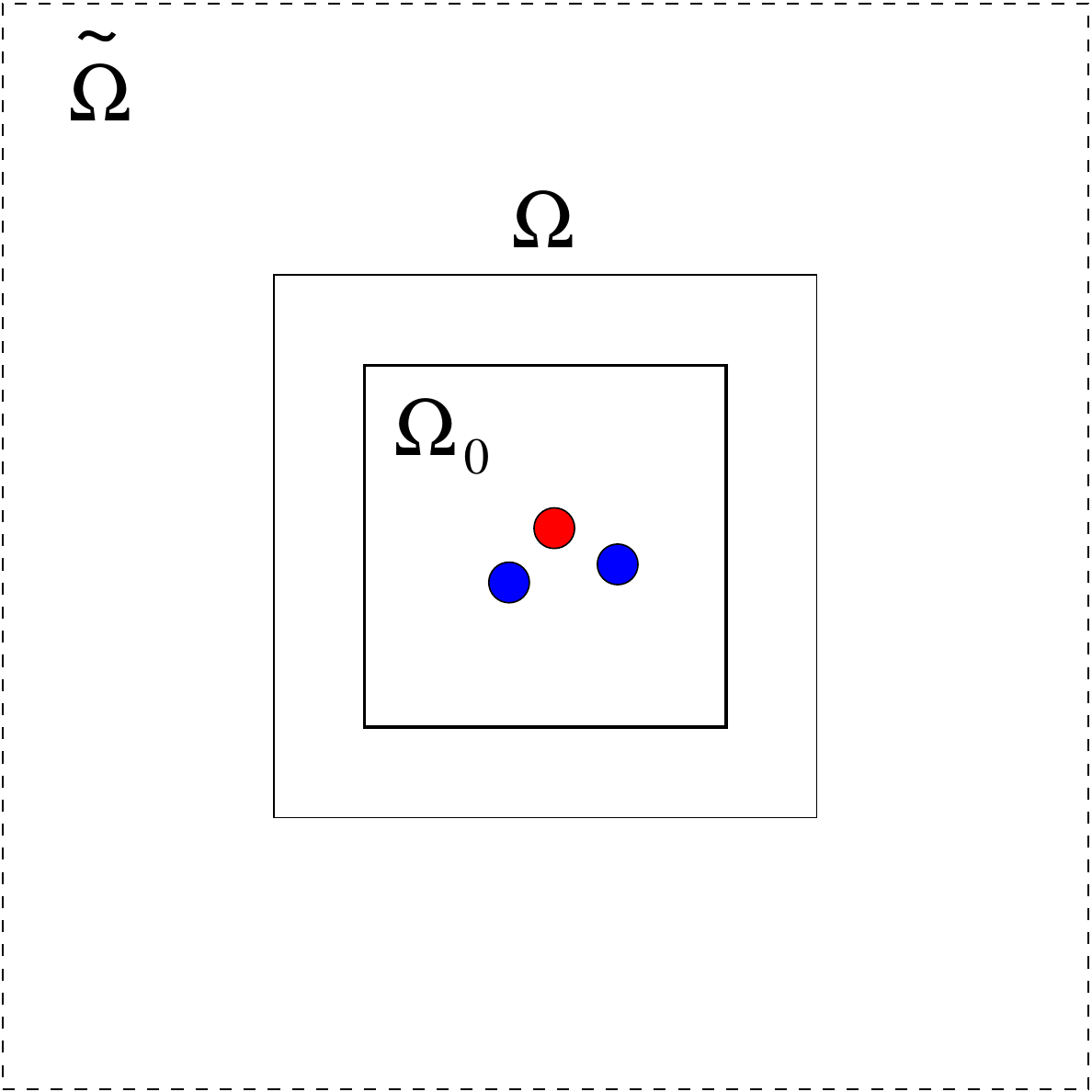}
\caption{\small $2$D unit cell $\Omega$, formation domain $\Omega_0$, and the 
auxiliary bounding box $\widetilde{\Omega}$.}
\label{fig:Unit_cell}  
\end{figure}

Similar to (\ref{eqn:sinc_general}), we introduce the auxiliary ``reference tensor'' 
$\widetilde{\bf P}_R \in \mathbb{R}^{\widetilde{n}\times \widetilde{n} \times \widetilde{n}}$,
living on the grid $\Omega_{\widetilde{n}}$  and approximating the Newton 
kernel in $\widetilde{\Omega}$,
\begin{equation} \label{eqn:master_pot}
\widetilde{\bf P}_R= 
\sum\limits_{q=1}^{R} \widetilde{\bf p}^{(1)}_q \otimes 
\widetilde{\bf p}^{(2)}_q \otimes \widetilde{\bf p}^{(3)}_q
\in \mathbb{R}^{\widetilde{n}\times \widetilde{n} \times \widetilde{n}}.
\end{equation}

The core potential $v_c(x)$ for a single molecule is approximated by a weighted 
sum of canonical tensors
\begin{equation} \label{eqn:core_tensSum}
 {\bf P}_{c} = \sum_{\nu=1}^{M_0} Z_\nu {\bf P}_{{c},\nu}\approx \widehat{\bf P}_{c} 
\in \mathbb{R}^{n\times n \times n},
\end{equation}
where the rank-$R$ tensor ${\bf P}_{{c},\nu}= = {\cal W}_{\nu} \widetilde{\bf P}_R$ 
represents the single reference Coulomb potential in the form (\ref{eqn:master_pot})
shifted and restricted to $\Omega_{n}$ via the windowing operator
${\cal W}_{\nu}={\cal W}_{\nu}^{(1)}\otimes  {\cal W}_{\nu}^{(2)}\otimes {\cal W}_{\nu}^{(3)}$, 
\cite{VeBoKh:Ewald:14},
\begin{equation} \label{eqn:core_tens}
 {\bf P}_{c,\nu} = {\cal W}_{\nu} \widetilde{\bf P}_R =  
\sum\limits_{q=1}^{R} {\cal W}_{\nu}^{(1)} \widetilde{\bf p}^{(1)}_q \otimes 
{\cal W}_{\nu}^{(2)} \widetilde{\bf p}^{(2)}_q 
\otimes {\cal W}_{\nu}^{(3)} \widetilde{\bf p}^{(3)}_q\in \mathbb{R}^{n\times n \times n}.
\end{equation}
Here every rank-$R$ canonical tensor 
${\cal W}_{\nu} \widetilde{\bf P}_R \in \mathbb{R}^{n\times n \times n}$, $\nu=1,...,M_0$,
is understood as a sub-tensor of the reference tensor 
obtained by a shift and restriction (windowing) of $\widetilde{\bf P}_R$ onto the $n \times n \times n$ 
grid $\Omega_{n}$ in the unit cell $\Omega$, $\Omega_{n} \subset \Omega_{\widetilde{n}}$. 
A shift from the origin is specified according to the coordinates of the corresponding nuclei, $a_\nu$,
counted in the $h$-units.

The initial rank bound $rank({\bf P}_{c})\leq M_0 R$ 
for the direct sum of canonical tensors in (\ref{eqn:core_tensSum}) can be 
improved (see \cite{VeBoKh:Ewald:14}, Remark 2.2).
In the following, we denote by $\widehat{\bf P}_{c}$ the rank-$R_c$ ($R_c\leq M_0 R$) 
canonical tensor obtained from ${\bf P}_{c}$ by the rank optimization 
procedure subject to certain threshold 
(in numerical tests we have $R_c \approx R$).

For the tensor representation of the Newton potentials, ${\bf P}_{{c},\nu}$, we make use 
of the piecewise constant discretization on the equidistant tensor grid, 
where, in general, the univariate grid size $n$ can be noticeably smaller 
than that used for the piecewise linear discretization applied to the Laplace operator.
Indeed, the Galerkin approximation to the eigenvalue problem is constructed
by using the global basis functions (reduced basis set $\{{g}_k\}$, $k=1,...,m_0$), hence
the grid-based representation of these basis functions can be different 
in the calculation of the kinetic and potential parts in the Fock operator.
The grid size $n$ is the only controlled by the 
approximation error for the integrals in (\ref{eqn:Core_Ham}) and by the numerical efficiency 
depending on the separation rank parameters.

Given tensor ${\bf P}_{c}$, the entries in the stiffness matrix $V_c$ in (\ref{eqn:Core_Ham_V}) 
can be evaluated by simple tensor operations.
Fixed the GTO-type basis set $\{{g}_k\}$, $k=1,...,m_0$, i.e. $N_b=m_0$, defined in the 
scaled unit cell ${\Omega}$, where for ease of presentation functions ${g}_k$ are supposed to be separable.
Introduce the corresponding rank-$1$ coefficients tensors 
${\bf G}_k={\bf g}_k^{(1)}\otimes{\bf g}_k^{(2)} \otimes{\bf g}_k^{(3)}$ 
representing their piecewise constant approximations $\{\widehat{g}_k\}$ 
on the fine ${n}\times {n}\times {n}$ grid. 
Then the entries of 
the respective Galerkin matrix ${V}_c=[{v}_{km}]$ in (\ref{eqn:Core_Ham_V}) 
approximating the core potential operator $v_c$ in (\ref{eqn:ElectrostPot})
 are represented (approximately) by the following tensor operations,
\begin{equation}  \label{eqn:nuc_pot}
 {v}_{km} \approx \int_{{\Omega}} V_c(x) \overline{g}_k(x) \overline{g}_m(x) dx 
 \approx  \langle {\bf G}_k \odot {\bf G}_m ,  \widehat{\bf P}_{c}\rangle =: {V}_{km} , 
\quad 1\leq k, m \leq m_0.
\end{equation}

The error arising due to the separable $\varepsilon$-approximation 
of the discretized nuclear potential 
is controlled by the rank parameter $R_{c}= rank(\widehat{\bf P}_{c})$. Now
letting $rank({\bf G}_m) = 1$ implies that each matrix element is to be computed with 
linear complexity in the univariate grid-size $n$, $O( R_{c} \, n)$. 
The almost exponential convergence of the tensor approximation 
in the separation rank $R_{c}$ leads to the asymptotic behavior of the
$\varepsilon$-rank, $R_{c}=O(|\log \varepsilon |)$.

\subsection{Nuclear potential operator for a lattice system in a box}
\label{ssec:Core_Ham_gener}

Here we apply the previous constructions to the lattice structured location of nuclei.
Given the potential sum $v_c(x)$ defined by (\ref{eqn:ElectrostPot})
in the scaled unit cell $\Omega = [-\frac{b}{2},\frac{b}{2}]^3$ of size $b\times b \times b$,
see Figure \ref{fig:Unit_cell}, 
we specify the smaller subdomain  $\Omega_0= [-\frac{b_0}{2},\frac{b_0}{2}]^3 \subset \Omega$ 
(called the formation cell) whose interior 
contains all atomic centers in the unit cell included into the summation 
in (\ref{eqn:ElectrostPot}).

Let us consider an interaction potential in a symmetric computational box (supercell)  
$$
\Omega_L =B_1\times B_2 \times B_3, \quad \mbox{with} 
\quad B_\ell = \frac{1}{2}[- b_0 L_\ell -b  ,b_0L_\ell + b ]
$$ 
consisting of a union of $L_1 \times L_2 \times L_3$ unit cells $\Omega_{\bf k}$,
obtained  by a shift of the reference domain $\Omega$ along the lattice vector 
${\bf b_0}/2 + b_0 {\bf k}$, where
${\bf k}=(k_1,k_2,k_3)\in \mathbb{Z}^3$, such that for $\ell=1,2,3$,
\[
 k_\ell \in {\cal K}_\ell:=\{0,1,...,L_\ell-1\}.
\]
In this notation the choice $L_\ell=1$ corresponds to the 3D one-layer system 
in the respective variable as illustrated in Figure \ref{fig:Hydro_chain115_3D}.
Figure \ref{fig:Lattice_5_2D} represents the $2$D projection of the 3D computational domain
$\Omega_L$ for the $L_1 \times 1 \times 1  $ molecular chain with $L_1=5$. Dashed regions
correspond to the overlapping parts between shifted unit cells.

\begin{figure}[htbp]
\centering
\includegraphics[width=7.0cm]{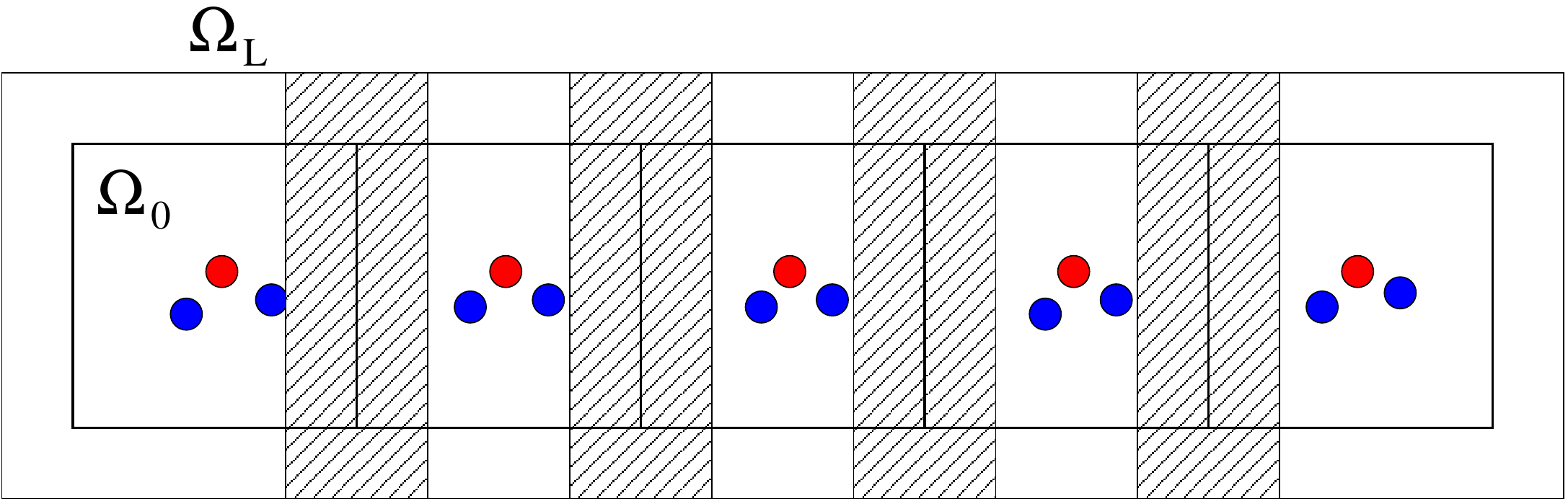}
\caption{\small  $2$D projection of the supercell for the $5 \times 1 \times 1$ chain in $3$D.}
\label{fig:Lattice_5_2D}  
\end{figure}
Figure \ref{fig:Hydro_chain115_3D} represents the geometry of the $3$D chain-type 
computational "tube" $\Omega_L$.

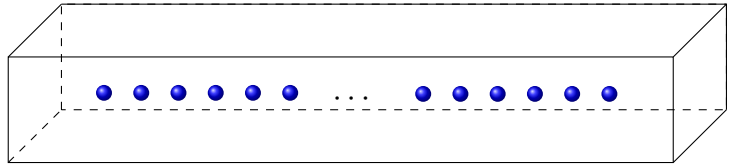
\begin{figure}[htbp]
\centering
\label{fig:Hydro_chain115_3D}
\begin{tikzpicture}[scale=0.7]

\path[draw] (1.5,4) -- (14,4);  
\path[draw] (14,2) -- (14,4); 
 
\path[draw] (0.5,3) -- (1.5,4);  
\path[draw] (13,3) -- (14,4);  

\path[draw] (13,1) -- (14,2);  
\path[draw][style=dashed] (0.5,1) -- (1.5,2);  

\draw (0.5,1) rectangle (13,3);
\draw[style=dashed] (1.5,2) rectangle (14,4);
 
\foreach \x in {2,2.7,3.4,4.1,4.8,5.5} {
\shade[shading = ball, ball color = blue] (\x+0.3,2.32) circle (.15);
}
 
\node at (7,2.2) {$\cdots$};

 \node at (0,0) { };
 \node at (14,5) { };
\foreach \x in {8,8.7,9.4,10.1,10.8,11.5} {
\shade[shading = ball, ball color = blue] (\x+0.3,2.3) circle (.15);
}
\end{tikzpicture}
\caption{Example of the $L \times 1 \times 1$ chain in $3$D.}
\end{figure}

For the discussion of complexity issues, we  often consider a cubic lattice of equal 
sizes $ L_1= L_2=L_3=L$.
By the construction, we set $b=n h$ and $b_0=n_0 h$, where the mesh-size $h >0$ is 
chosen the same for all spacial variables. 

In the most interesting case of extended system in a box, further called case (B),
the potential $v_{c_L}(x)$, for $x\in \Omega_{L}$, 
is obtained by summation over all unit cells $\Omega_{\bf k}$ in $\Omega_L$,
\begin{equation}\label{eqn:EwaldSumE}
v_{c_L}(x)=  \sum_{\nu=1}^{M_0} \sum\limits_{{\bf k}\in {\cal K}^3} 
\frac{Z_\nu}{\|{x} -a_\nu - b {\bf k} \|}, \quad x\in \Omega_{L}. 
\end{equation}
Note that the direct calculation by (\ref{eqn:EwaldSumE}) is performed at each of $L^3$ unit cells 
$\Omega_{\bf k}\subset \Omega_L$, ${\bf k}\in {\cal K}^3$, on a 3D lattice, 
which presupposes substantial numerical costs at least of the order of $O(L^3)$ per unit cell.

The fast calculation of (\ref{eqn:EwaldSumE}) is implemented by using the tensor summation method
introduced in \cite{VeBoKh:Ewald:14,VeKhorEwTuck_NLLA:15} which can be  described as follows.
Let $\Omega_{N_L}$ be the $N_L\times N_L\times N_L$ uniform grid on $\Omega_L$ with the 
same mesh-size $h$
as above, and introduce the corresponding space of piecewise constant basis functions 
of the dimension $N_L^3$, where we have $N_L = n_0 L + n-n_0$. Given the reference tensor
in (\ref{eqn:master_pot}), the resultant lattice sum is presented by the canonical tensor 
${\bf P}_{c_L}$ 
\begin{equation}\label{eqn:EwaldTensorGl}
{\bf P}_{c_L}= 
\sum\limits_{\nu=1}^{M_0} Z_\nu   \sum\limits_{q=1}^{R}
(\sum\limits_{k_1\in {\cal K}_1} {\cal W}_{\nu({k_1})} \widetilde{\bf p}^{(1)}_{q}) \otimes 
(\sum\limits_{k_2\in {\cal K}_2} {\cal W}_{\nu({k_2})} \widetilde{\bf p}^{(2)}_{q}) \otimes 
(\sum\limits_{k_3\in {\cal K}_3} {\cal W}_{\nu({k_3})} \widetilde{\bf p}^{(3)}_{q}),
\end{equation}
whose  rank is uniformly bounded $R_c \leq M_0 R$.
The numerical cost and storage size are bounded by $O(M_0 R L N_L )$, 
and $O(M_0 R N_L)$, respectively (see \cite{VeBoKh:Ewald:14}, Theorem 3.1),
where  $N_L= O(n_0 L)$. The lattice sum in (\ref{eqn:EwaldTensorGl}) converges 
only conditionally as $L\to \infty$.
This aspect will be addressed in Section \ref{ssec:Complexity_EigPr} following the approach 
discussed in \cite{VeBoKh:Ewald:14,VeKhorEwTuck_NLLA:15}.

In the case of lattice system in a box, we define the basis set on a supercell $\Omega_{L}$ 
(and on $\widetilde{\Omega}_L$) 
by translation of the generating basis, defined in $\Omega_0$ for the single molecule, 
by the lattice vector $b {\bf k}$, i.e., 
$\{g_{\mu}({x})\} \mapsto \{g_{\mu}({x+ b {\bf k} })\}$, $\mu=1,...,m_0$,
where ${\bf k}=(k_1,k_2,k_3)\in {\cal K}^3$, 
assuming zero extension 
of $\{g_{\mu}({x+ b {\bf k} })\}$ beyond each local bounding box $\widetilde{\Omega}_{\bf k}$.
The corresponding tensor representation of such functions is denoted by ${\bf G}_{{\bf k},\mu}$. 
The total number of basis functions for the lattice system is equal to $N_b=m_0 L^3$.

In what follows, the matrix block entries of the $N_b\times N_b$ stiffness matrix $V_{c_L}$, 
corresponding to large basis set on a supercell $\Omega_{L}$,
will be numbered by a pair of multi-indices, 
$V_{c_L}=[V_{{\bf k}{\bf m}}]$, where each $m_0\times m_0$ 
matrix block $V_{{\bf k}{\bf m}}$ is defined by
\begin{equation} \label{eqn:nuc_MatrSparsP}
V_{{\bf k}{\bf m}}(\mu,\nu) = \langle {\bf G}_{{\bf k},\mu} \odot {\bf G}_{{\bf m},\nu},{\bf P}_{c_L}\rangle, 
\quad {\bf k}, {\bf m}\in {\cal K}^3.
\end{equation}
This definition introduces the three-level block structure in the matrix $V_{c_L}$, which
will be discussed in what follows. 

In the practically interesting case of localized atomic orbitals (AO) basis, 
the matrix  $V_{c_L}$ exhibits the banded block 
sparsity pattern since the effective support of localized AO 
associated with every unit cell $\Omega_{\bf k} \subset \widetilde{\Omega}_{\bf k}$ 
overlaps only fixed (small) number of neighboring cells.
We call the number of overlapping neighboring cells by 
the {\it overlap constant}, $L_0$. 
The constant $L_0$ measures the essential overlap between basis functions in each spacial direction. 
For example, Figure \ref{fig:Lattice_5_2D} corresponds to the choice $L_0=2$.

\begin{lemma}\label{lem:SparseCaseE}
Assume that  the overlap constant does not exceed $L_0$, then: 

(a) The number of non-zero blocks in each block row (column) of the symmetric 
Galerkin matrix $V_{c_L}$ does not exceed $(2 L_0 + 1)^3$.
 
(b) The storage size is bounded by $m_0^2 [(L_0 + 1)L]^3$.

(c) The cost for evaluation of each $m_0\times m_0$ matrix block is bounded by $O(m_0^2 M_0 R N_L)$.
\end{lemma}
\begin{proof}
In case (B), the matrix elements of $V_{c_L}=[v_{km}]\in \mathbb{R}^{N_b\times N_b}$ 
represented in (\ref{eqn:nuc_pot}), or in the block form in (\ref{eqn:nuc_MatrSparsP}),
can be expressed  by the following tensor operations
 \begin{equation} \label{nuc_potMatrTot}
 v_{{k}{ m}} =  \int_{\mathbb{R}^3} v_c(x) \overline{g}_k(x) \overline{g}_m(x) dx 
\approx  \langle {\bf G}_k \odot {\bf G}_m ,   {\bf P}_{c_L}\rangle =: v_{km}, 
\quad 1\leq k, m \leq N_b,
\end{equation}
where again $\{\overline{g}_k\}$  denotes the piecewise constant representations to 
the respective Galerkin basis functions.
This leads to the block representation (\ref{eqn:nuc_MatrSparsP}) 
in terms of univariate vector operations
\[
\begin{split}
 V_{{\bf k}{\bf m}} & = \sum\limits_{\nu=1}^{M_0} Z_\nu \sum\limits_{q=1}^{R}
\langle {\bf G}_{\bf k} \odot {\bf G}_{\bf m} , 
(\sum\limits_{k_1 \in {\cal K}} {\cal W}_{\nu({k_1})} \widetilde{\bf p}^{(1)}_{q}) \otimes 
(\sum\limits_{k_2\in {\cal K}} {\cal W}_{\nu({k_2})} \widetilde{\bf p}^{(2)}_{q}) \otimes 
(\sum\limits_{k_3\in {\cal K}} {\cal W}_{\nu({k_3})} \widetilde{\bf p}^{(3)}_{q})  \rangle \\
 &=
\sum\limits_{\nu=1}^{M_0} Z_\nu \sum\limits_{q=1}^{R}
\prod\limits_{\ell=1}^3
\langle {\bf g}_{\bf k}^{(\ell)} \odot { \bf g}_{\bf m}^{(\ell)},
\sum\limits_{k_\ell\in {\cal K}} {\cal W}_{\nu({k_\ell})} \widetilde{\bf p}^{(\ell)}_{q} \rangle.  
\end{split}
\]
Combining  the block representation (\ref{eqn:nuc_MatrSparsP}) and taking into account 
the overlapping property
\begin{equation} \label{eqn:Overlap_Basis}
 {\bf G}_{\bf k} \odot {\bf G}_{\bf m}=0 \quad \mbox{if} \quad | k_\ell - m_\ell| \geq L_0, 
\end{equation}
we are able to analyze the block sparsity pattern in the Galerkin matrix $V_{c_L}$.
Given $3 M_0 R $ canonical vectors 
$\sum\limits_{k_\ell \in {\cal K}} {\cal W}_{\nu({k_\ell})} 
\widetilde{\bf p}^{(\ell)}_{q}\in \mathbb{R}^{N_L}$, 
where $N_L$ denotes the total number of grid points in $\Omega_L$ in each space variable.
Now the numerical cost to compute ${v}_{km}$ for every fixed index $(k,m)$ is 
estimated by $O(M_0 R  N_L)$ 
indicating linear scaling in the large grid parameter $N_L$ (but not cubic). 

Fixed the row index in $(k,m_\ast)$, then item (b) follows from the bound on the total number of 
overlapping cells $\Omega_{\bf k}$ in the effective integration 
domain in (\ref{nuc_potMatrTot}), that is $(2 L_0 + 1)^3$,  
and from the symmetry of $V_{c_L}$.
\end{proof}

Figure \ref{fig:3DCoreHamPer1} illustrates the sparsity pattern of the 
nuclear potential contribution $V_{c_L}$ in the matrix $H$, computed 
for $L\times  1\times 1$ lattice 
in a $3$D supercell with $L=48$ and $m_0=4$, and
the overlapping parameter $L_0=3$. In Figure (\ref{fig:3DCoreHamPer1}), right one 
can observe the nearly-boundary effects due to 
the non-equalized  contributions from the left and from the right (supercell in a box). 

\begin{figure}[htbp]
\centering
\includegraphics[width=6.0cm]{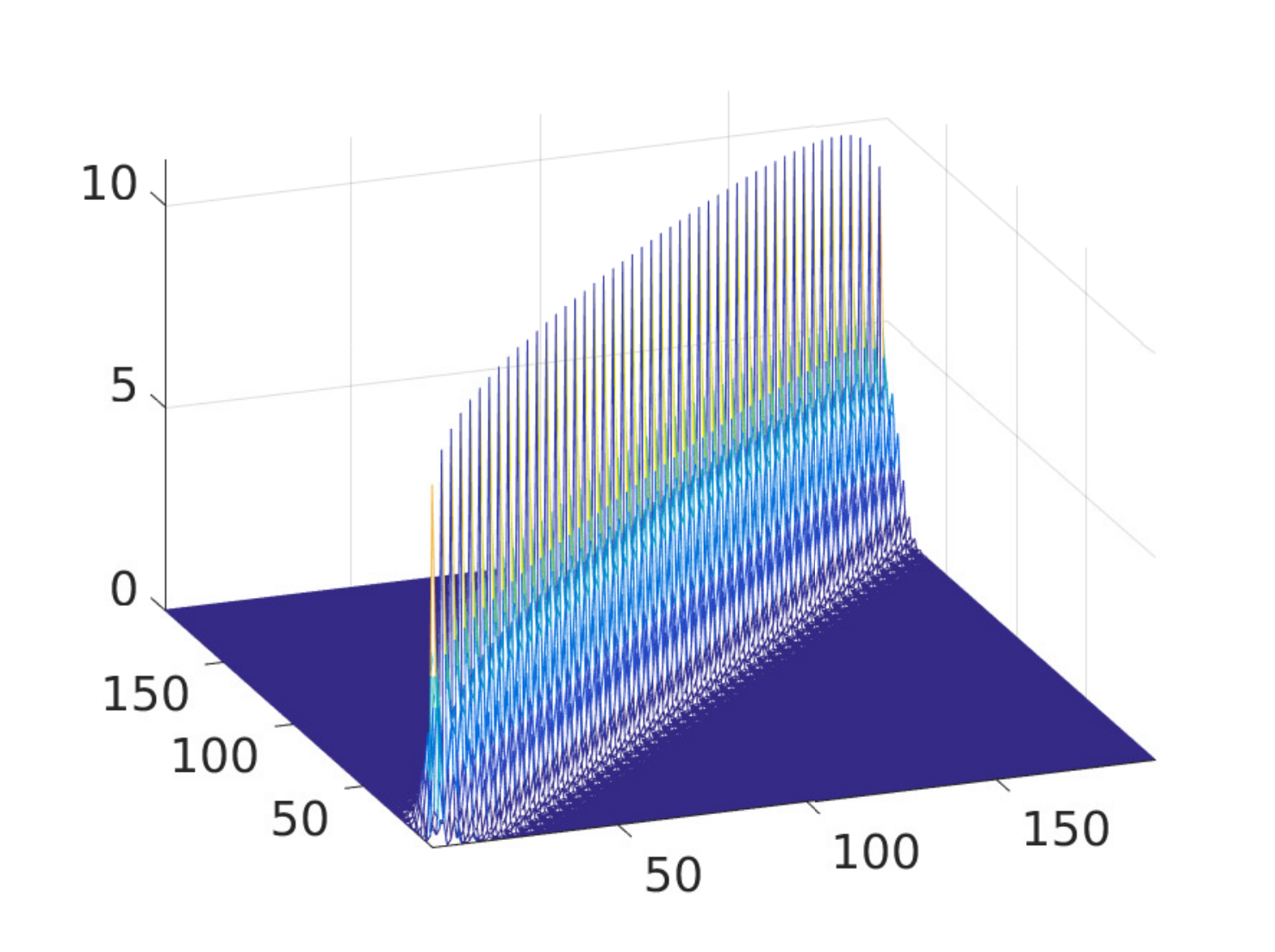}\quad  
\includegraphics[width=6.0cm]{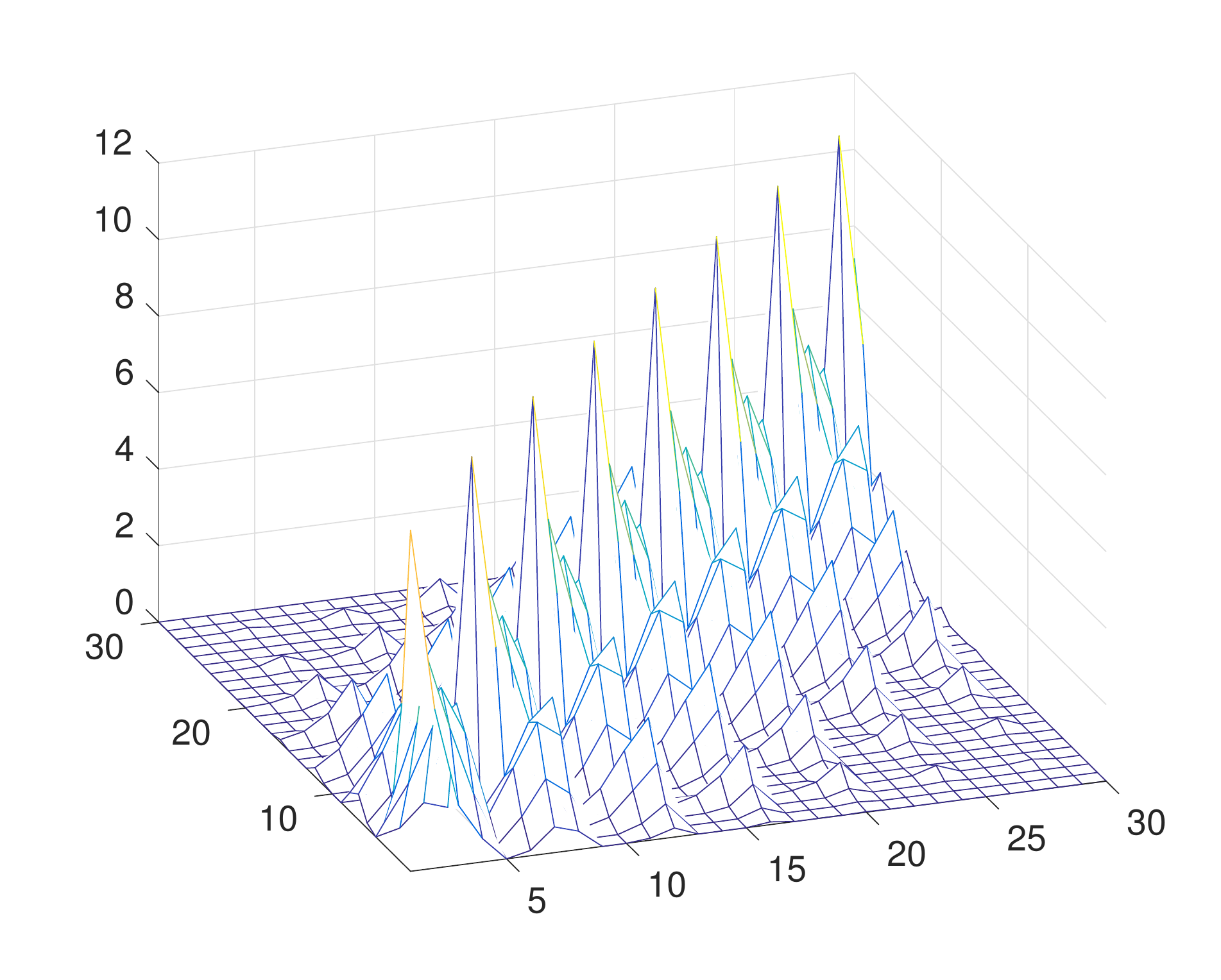}
\caption{\small  Block-sparsity in the matrix $V_{c_L}$, for a finite lattice 
$L\times 1 \times 1 $ with $L=48$ (left); 
zoom of the first $30\times 30 $ entries of the matrix (right).}
\label{fig:3DCoreHamPer1}  
\end{figure}

Notice that the quantized tensor approximation (QTT) of canonical vectors involved in ${\bf G}_k$ 
and ${\bf P}_{c_L}$ reduces this cost to the logarithmic scale, $O(M_0 R \log N_L)$, 
that is important in the case of large $L$ in view of $N_L=O(L)$, see 
the discussion in \cite{VeBoKh:Ewald:14}.

The block $L_0$-diagonal structure of the matrix $V_{c_L}=[V_{{\bf k}{\bf m}}]$,
${\bf k},{\bf m}\in {\cal K}^3$ described by Lemma \ref{lem:SparseCaseE} allows 
the essential saving in the storage costs. 

However, the polynomial complexity scaling in $L$ leads to severe limitations on the number of
unit cells. These limitations can be relax if we look more precisely on the
defect between matrix ${V}_{c_L}$ and its block-circulant version corresponding to
the periodic boundary conditions (see \S\ref{ssec:Core_Ham_period}). 
This defect can be split into two components corresponding to their local and non-local features:
\begin{enumerate}
 \item [(A)] The non-local effect indicates the asymmetry in the interaction potential 
 sum on the lattice in a box.
\item [(B)] The near boundary (local) defect  effects only those blocks in 
 $V_{c_L}=\{V_{{\bf k}{\bf m}}\}$ lying in the $L_0$-width of $\partial\Omega_L$.
\end{enumerate}

The defect in item (A) can be diminished by a slight  modification of the core potential 
to the shift invariant Toeplitz-type form $V_{{\bf k}{\bf m}} = V_{|{\bf k}-{\bf m}|}$
by replication of the central unit cell to the whole lattice,
as considered in Section 3. In this way the overlap condition 
(\ref{eqn:Overlap_Basis}) for the tensor ${\bf G}_{\bf k}$ 
will impose the $(2L_0 +1)$ block diagonal sparsity in the block-Toeplitz matrix.

The boundary effect in item (B) becomes relatively small for large number of unit cells
so that the block-circulant part of the matrix $V_{c_L}$ is getting dominating
as $L\to \infty$. 

The full diagonalization for the above mentioned matrix $V_{c_L}$ can be prohibitively expensive.
However, the efficient storage and fast matrix-vector multiplication algorithms 
can be applied in the framework of structured iteration on subspace for calculation of a small
subset of eigenvalues, see \cite{BeDoKh2_BSE2:16}.

\subsection{Discrete Laplacian  and the mass matrix}\label{ssec:Lap_Op}

In the case of a single molecule, the Laplace operator in (\ref{eqn:EIGHFcore}), (\ref{eqn:HFcore}) 
is posed in the unit cell $ \Omega=[-b/2,b/2]^3 \in \mathbb{R}^3 $,
subject to the homogeneous Dirichlet boundary conditions on $\partial \Omega$.
Periodic case corresponds to periodic boundary conditions.
Given discretization parameter $\widehat{n} \in \mathbb{N}$,  we use
the equidistant $\widehat{n}\times \widehat{n} \times \widehat{n}$ tensor 
grid $\Omega_{\widehat{n}}=\{x_{\bf i}\} $, 
${\bf i} \in {\cal I} :=\{1,...,\widehat{n}\}^3 $,
defined by the mesh-size $h=b/(\widehat{n} + 1)$.
This grid might be different from $\Omega_{n} $ introduced in \S2.2
for representation of the interaction potential in the set of piecewise constant basis functions 
(usually, $n\leq \widehat{n}$).

Define a linear tensor-product interpolation operator ${\bf I}$ via the set of product hat functions,
 $\{\xi_{\bf i} :=  \xi_{i_1} (x_1) \xi_{i_2} (x_2)\xi_{i_3} (x_3)$, 
${\bf i} \in {\cal I}\}$,
associated with the respective grid-cells in $\Omega_{\widehat{n}}$.
Here the linear interpolant ${\bf I}= {I}_1\times {I}_1 \times {I}_1$ is 
a product of 1D interpolation operators, 
where ${I}_1:C^0([-b,b])\to W_h:=span\{\xi_i\}_{i=1}^{\widehat{n}}$ 
is defined over the set of piecewise linear basis functions by 
$$
({I}_1 \, w)(x_\ell):=\sum_{{i}=1}^{\widehat{n}} w(x_{i_\ell})\xi_{i}(x_\ell), 
\quad  x_{\bf i} \in \Omega_{\widehat{n}}, \quad \ell=1,2,3.
$$

Define the 1D FEM Galerkin stiffness (for Laplacian) and mass matrices 
$A^{(\ell)}, S^{(\ell)} \in \mathbb{R}^{\widehat{n}\times \widehat{n}}$, respectively, by
\[
 A^{(\ell)} := \{ \langle \frac{d}{d x_\ell} \xi_i(x_\ell) , \frac{d}{d x_\ell} \xi_j(x_\ell) 
\rangle \}^{\widehat{n}}_{i,j=1} = \frac{1}{h} \mbox{tridiag} \{-1,2,-1\},
\]
\[ 
 S^{(\ell)}=\{ \langle \xi_i ,\xi_j\rangle \}^{\widehat{n}}_{i,j=1} = \frac{h}{6}\;
 \mbox{tridiag} \{1,4,1\},\quad \ell=1,\,2,\,3.
\]
For fixed dimension $d$ and $k\leq d$, introduce the mixed Kronecker product of matrices
$S^{(\ell)}$ and $A^{(\ell)}$
$$
\otimes_{(d \curlyvee k)}(S^{(\ell)},A^{(k)})= 
S^{(1)} \otimes ... \otimes S^{(k-1)} \otimes A^{(k)}\otimes S^{(k+1)}\otimes ...\otimes S^{(d)}.
$$
In the following, we apply the similar notations with respect to 
the Hadamard product of matrices $\odot$ and the usual multiplication operation, $\prod$.

Following \cite{KhorVBAndrae:12}, the rank-$3$ Kronecker tensor representation 
of the standard FEM Galerkin 
stiffness matrix for the Laplacian, $A_3\in 
\mathbb{R}^{\widehat{n}^{ 3}\times \widehat{n}^{3}}$, in the separable basis 
$\{\xi_i(x_1) \xi_j (x_2)\xi_k (x_3) \} $, $i,j,k = 1,\ldots \widehat{n}$,
reads as
$$
 A_3 := A^{(1)} \otimes S^{(2)} \otimes S^{(3)} + S^{(1)} \otimes A^{(2)} \otimes S^{(3)}
+ S^{(1)} \otimes S^{(2)} \otimes A^{(3)}\equiv 
\sum_{k=1}^d \otimes_{(d \curlyvee k)}(S^{(\ell)},A^{(k)}). 
$$ 
In turn, the mass matrix takes the separable Kronecker product form
\[
 S_3= S^{(1)} \otimes S^{(2)} \otimes S^{(3)}\in 
\mathbb{R}^{\widehat{n}^{3}\times \widehat{n}^{ 3}}.
\]


For given GTO-type basis set $\{g_k(x)=g_{k}^{(1)}(x_{1})g_{k}^{(2)}(x_{2})g_{k}^{(3)}(x_{3})\}$ 
define a set of piecewise linear basis functions 
$\widehat{g}_k^{(\ell)} := {I}_1 g_k^{(\ell)} $, $k=1,...,m_0$, and introduce
the separable grid-based approximation of the initial basis functions $g_k(x)$,
\begin{equation*}\label{eq. Gaus pwl}
g_k (x) \approx \widehat{g}_k (x) := \prod^3_{\ell=1} 
\widehat{g}_k^{(\ell)} (x_{\ell})=\prod^3_{\ell=1} 
\sum\limits^{\widehat{n}}_{i=1} g_{k}^{(\ell)}(x_{i_\ell}) \xi_i (x_{\ell}).
\end{equation*}
Here the rank-$1$ coefficients tensor 
${\bf G}_k={\bf g}_k^{(1)} \otimes {\bf g}_k^{(2)} \otimes {\bf g}_k^{(3)}
\in \mathbb{R}^{{\widehat{n}}^{\otimes 3}}$ given by the canonical vectors 
${\bf g}_k^{(\ell)}=\{g_{k}^{(\ell)}(x_{i_\ell})\}$, ($k=1,...,m_0$) 
is associated with the Kronecker product of vectors, 
${\bf g}_k=vec({\bf G}_k)\in \mathbb{R}^{\widehat{n}^{3}} $.
Let us agglomerate vectors ${\bf g}_k$
in a Kronecker product matrix $G={G}^{(1)} \otimes {G}^{(2)} \otimes {G}^{(3)}
\in  \mathbb{R}^{{\widehat{n}}^{3}\times m_0}$, where 
${G}^{(\ell)}=[{\bf g}_1^{(\ell)},...,{\bf g}_{m_0}^{(\ell)}]\in  
\mathbb{R}^{{\widehat{n}}\times m_0}$, ($\ell=1,2,3$), 
is constructed by concatenation of vectors ${\bf g}_k^{(\ell)}$.
Then
the Galerkin stiffness matrix for the Laplacian and the mass matrix in 
the GTO basis set $\{{\bf G}_k \}$ can be written as  
\begin{equation}\label{eqn:Lapl_Mass_Gal}
 A_G= G^T A_3 G\in  \mathbb{R}^{m_0\times m_0}, \quad 
 S_G=G^T S_3 G  \in  \mathbb{R}^{m_0 \times m_0},
\end{equation}
corresponding to the standard matrix-matrix transform under the change of basis.

Applying the above representations to the $L\times L \times L$ lattice systems 
as described in \S2.3
leads to the symmetric and sparce block-Toeplitz structure of the $N_b \times N_b$ Galerkin 
matrices with the block size $m_0\times m_0$ and with $N_b=m_0 L^3$.

\begin{proposition}\label{prop:matr_GTO_Laplace}
Assume that  the overlap constant does not exceed $L_0$, then: 

(A) The cost for evaluation of each $m_0\times m_0$ matrix block is bounded by
$O(m_0^2 \widehat{n})$.

(B) The number of non-zero blocks in each block row (column) of the symmetric 
Galerkin matrices $A_G$ and $S_G$ does not exceed $(2 L_0 + 1)^3$.
 
(C) Both $A_G$ and $S_G$ are symmetric $3$-level block-Toeplitz matrices.
The storage size is bounded by $m_0^2 (L_0 + 1)^3 L^3$.
\end{proposition}
\begin{proof}
First, notice that the matrix entries in $A_G=\{a_{k m}\}$ and 
$S_G=\{s_{k m}\}$, ($k,m = 1,...,m_0$) can be represented in the product form.
For example, we have
 \[
 s_{k m}= \langle S_3 {\bf g}_k, {\bf g}_m\rangle = 
 {\prod}_{\ell=1}^3 {{\bf g}_m^{(\ell)}}^T S^{(\ell)} {\bf g}_k^{(\ell)}.
\]

Combining this representation with (\ref{eqn:Lapl_Mass_Gal}) 
implies the matrix factorization 
\begin{equation}\label{eqn:Mass_Gal_fact}
S_G=G^T (S^{(1)} \otimes S^{(2)} \otimes S^{(3)}) G=  
({G^{(1)}}^T S^{(1)} {G}^{(1)})\odot ({G^{(2)}}^T S^{(2)}{G}^{(2)})
\odot ({G^{(3)}}^T S^{(3)}{G}^{(3)}),
\end{equation}
where $\odot$ means the Hadamard product of matrices.
The similar $d$-term sum of products representing  matrix elements 
$a_{k m}= \langle A_3 {\bf g}_k, {\bf g}_m\rangle$,
\[
\langle A_3 {\bf g}_k, {\bf g}_m\rangle = \sum_{p=1}^d {\prod}_{(d \setminus p)}
({{\bf g}_m^{(\ell)}}^T S^{(\ell)}{\bf g}_k^{(\ell)},{{\bf g}_m^{(p)}}^T A^{(p)} {\bf g}_k^{(p)}),
\]
leads to the $d$-term factorized representation of $A_G$ (say, $d=3$),
 \begin{equation} \label{eqn:Lapl_Gal_fact}
 A_G= \sum_{k=1}^3 \odot_{(d \setminus k)}
 ({G^{(\ell)}}^T S^{(\ell)} {G}^{(\ell)},{G^{(k)}}^T A^{(k)} {G}^{(k)} ).
\end{equation}
This proves the numerical cost for the matrix evaluation. Items (B) and (C) can be justified
by similar arguments as in Lemma \ref{lem:SparseCaseE}.
\end{proof}

\begin{remark}
Notice that in the periodic case both matrices, $A_G$ and $S_G$, possess  the
three-level block circulant structure as discussed in \S\ref{ssec:Core_Ham_period}
(see Appendix for definitions). 
\end{remark}

\section{Tensor factorization meets FFT block-diagonalization}
\label{sec:Core_Ham_period_FFT}

There are two basic approaches to mathematical modeling of the $L$-periodic molecular 
systems composed of $L\times L \times L$ elementary unit cells \cite{SzOst:1996}.  
In  the first approach, 
the system is supposed to contain an infinite set of equivalent atoms that 
map identically into itself under any translation by $L$ units in each spacial direction. 
The other model is based on the ring-type periodic structures 
consisting of $L$ identical units in each spacial direction, where every
unit cell of the periodic compound will be mapped to itself by applying a rotational 
transform  from the corresponding rotational group symmetry.

The main difference between these two concepts is in the treatment of the lattice sum 
of Coulomb interactions, thought, at the limit of $L\to \infty$ both models 
approach each other.
In this paper we mainly follow the first approach with the particular 
focus on the asymptotic complexity optimization  
for large lattice parameter $L$. The second concept is useful
for understanding the block structure of the Galerkin matrices for the Hartree-Fock 
operator.

The direct Hartree-Fock calculations for lattice structured systems 
in the localized GTO-type basis lead to the 
symmetric block circulant/Toeplitz matrices,    
where the first-level blocks, $A_0,...,A_{L-1}$, may have further block structures 
to be discussed in what follows (see also Appendix).
In particular, the Galerkin approximation to the 3D Hartree-Fock core Hamiltonian 
in periodic setting leads to the symmetric, three-level block circulant matrix, see 
\S\ref{ssec_Append:ML_block-circ} 
concerning the definition of multilevel block circulant (MBC) matrices.

\subsection{Block-diagonal form of the system matrix}
\label{ssec:Tensor_bdiag}

In this paragraph, we introduce the new data-sparse block structure  by imposing 
the low-rank tensor factorizations within the diagonalized  MBC matrix
in the matrix class ${\cal BC} (d,{\bf L},m_0)$, where ${\bf L}=(L_1,...,L_d)$.

The block-diagonal form of a MBC matrix is well known in the literature, see e.g. \cite{Davis}.
A diagonalization of a $d$-level MBC matrix is based on representation via a sequence 
of cycling permutation matrices 
$\pi_{L_1}, ...,\pi_{L_d}$, $d=1,2,3, ...$. Recall that the $d$-dimensional Fourier transform (FT) 
can be defined via the Kronecker product of the univariate FT matrices (Kronecker rank-$1$ operator), 
$$
F_{\bf L}=F_{L_1}\otimes \cdots \otimes F_{L_d}.
$$

Here we prove the diagonal representation in a form that is useful for the description
of tensor-based numerical algorithms. 
To that end we generalize the notations ${\cal T}_L$ and $\widehat{A}$ 
(see Appendix, \S \ref{ssec_Append:block-circ}) to the class of multilevel matrices.
We denote by $\widehat{A}\in \mathbb{R}^{|{\bf L}|m_0\times m_0}$ the first 
block column of a matrix $A\in {\cal BC} (d,{\bf L},m_0)$, with a shorthand notation 
$$
\widehat{A}=[A_0,A_1,...,A_{L_1-1}]^T,
$$ 
and define a $|{\bf L}|\times m_0 \times m_0$ tensor ${\cal T}_{\bf L} \widehat{A}$, which 
represents slice-wise all generating $m_0\times m_0$ matrix blocks in $\widehat{A}$ 
(reshaping of $\widehat{A}$).
Notice that in the case $m_0=1$, the matrix $\widehat{A}\in \mathbb{R}^{|{\bf L}|\times 1}$ 
represents the first column of $A$.
Now the Fourier transform $F_{\bf L}$ applies to ${\cal T}_{\bf L} \widehat{A}$ column-wise, 
while the backward reshaping of the resultant tensor, ${\cal T}_{\bf L}'$, returns 
an $|{\bf L}|m_0 \times m_0$ block matrix column. 
In the following we use the conventional matrix indexing and 
assume that the lattice ${\bf k}$-index runs as $k_\ell=0,1,..., L_\ell-1$.

\begin{lemma}\label{lem:DiagMLCirc}
A matrix $A\in {\cal BC} (d,{\bf L},m_0)$ can be converted to the block-diagonal form  by 
the Fourier transform $F_{\bf L}$,
\begin{equation} \label{eqn:DiagMLcirc}
A= (F_{\bf L}^\ast \otimes I_{m_0}) \operatorname{bdiag}_{m_0\times m_0} 
\{ \bar{A}_{\bf 0}, \bar{A}_{\bf 1},\ldots , 
\bar{A}_{\bf L-1}\}(F_{\bf L} \otimes I_{m_0}),
\end{equation} 
where
\[
  \left[ \bar{A}_{\bf 0}, \bar{A}_{\bf 1},\ldots , \bar{A}_{\bf L-1}\right]^T = 
 {\cal T}_{\bf L}'(F_{\bf L} ({\cal T}_{\bf L} \widehat{A})).
\]
\end{lemma}
\begin{proof} First, we confine ourself to the case of three-level matrices, i.e. $d=3$. 
We apply the Kronecker product decomposition (\ref{eqn:bcircPol}) successively to each 
level of the block-circulant $A$ to obtain (see (\ref{eqn:perShift}) for the definition of $\pi_{L}$)
\begin{align*}
A & = \sum\limits^{L_1 -1}_{k_1=0} \pi_{L_1}^{k_1} \otimes { A}_{k_1} \\  \nonumber 
  & = \sum\limits^{L_1 -1}_{k_1=0} \pi_{L_1}^{k_1}\otimes 
      (\sum\limits^{L_2 -1}_{k_2=0} \pi_{L_2}^{k_2}\otimes { A}_{k_1 k_2} )=
      \sum\limits^{L_1 -1}_{k_1=0}\sum\limits^{L_2 -1}_{k_2=0}
      \pi_{L_1}^{k_1}\otimes \pi_{L_2}^{k_2}\otimes { A}_{k_1 k_2}  \\ \nonumber 
  & = \sum\limits^{L_1 -1}_{k_1=0}\sum\limits^{L_2 -1}_{k_2=0}\sum\limits^{L_3 -1}_{k_3=0} 
       \pi_{L_1}^{k_1}\otimes \pi_{L_2}^{k_2}\otimes \pi_{L_3}^{k_3}\otimes A_{k_1 k_2 k_3},    \nonumber
\end{align*} 
where ${ A}_{k_1}\in \mathbb{R}^{L_2 L_3 m_0 \times L_2 L_3 m_0}$,
${ A}_{k_1 k_2} \in  \mathbb{R}^{L_3  m_0\times L_3 m_0}$
and $A_{k_1 k_2 k_3}\in \mathbb{R}^{m_0\times m_0}$. 

Diagonalizing the periodic shift matrices $\pi_{L_1}^{k_1}, \pi_{L_2}^{k_2}$, and 
$\pi_{L_3}^{k_3}$ via the 1D Fourier transform (see Appendix), we arrive at the block-diagonal
representation
\begin{align}\label{eqn:MLbcircDiag2}
A & = (F_{\bf L}^\ast \otimes I_{m_0}) \left[\sum\limits^{L_1 -1}_{k_1=0}\sum\limits^{L_2 -1}_{k_2=0}
      \sum\limits^{L_3 -1}_{k_3=0}
      D_{L_1}^{k_1}\otimes D_{L_2}^{k_2}\otimes D_{L_3}^{k_3}\otimes A_{k_1 k_2 k_3} \right]
 (F_{\bf L}\otimes I_{m_0})
      \\ 
  & = (F_{\bf L}^\ast \otimes I_{m_0}) 
 \mbox{bdiag}_{m_0\times m_0} \{{\cal T}_{\bf L}'(F_{\bf L} ({\cal T}_{\bf L} \widehat{A}))\} 
(F_{\bf L}\otimes I_{m_0}),\nonumber
 \end{align}
where the monomials of diagonal matrices $D_{L_\ell}^{k_\ell}\in \mathbb{R}^{L_\ell \times L_\ell}$, 
$\ell=1,2,3$ are defined by (\ref{eqn:diagshift}).
The generalization to the case $d >3$ can be proven by the similar argument.
\end{proof}

Taking into account representation (\ref{eqn:symBc}),
the multilevel symmetric block circulant matrix can be described in form 
(\ref{eqn:DiagMLcirc}), such that all real-valued diagonal blocks remain symmetric.

The following remark compares the properties of circulant and Toeplitz matrices. 
\begin{remark}\label{rem:BToepl}
A block Toeplitz matrix  does not allow explicit diagonalization by FT as it is
the case for block circulant matrices.
However, it is well known that a block Toeplitz matrix can be extended to the 
double-size (at each level)
block circulant that makes it possible the efficient matrix-vector multiplication, 
and, in particular, the efficient application of power method for finding its
senior eigenvalues.
\end{remark}

\subsection{Low-rank tensor structure within diagonalized block matrix}
\label{ssec:Tensor_bcirc}

In the particular case $d=3$, the general block-diagonal representation  (\ref{eqn:MLbcircDiag2})
allows the reduced storage cost for the coefficients tensor 
$[A_{k_1 k_2 k_3}]$ to the order of $O(|{\bf L}| m_0^2)$, where $|{\bf L}|=L_1 L_2 L_3$.
 Introduce the short notation 
$D_{\bf L}^{\bf k}=  D_{L_1}^{k_1}\otimes D_{L_2}^{k_2}\otimes \cdots \otimes D_{L_d}^{k_d}$,
then (\ref{eqn:MLbcircDiag2}) takes a form
\[
 A= (F_{\bf L}^\ast \otimes I_{m_0})(\sum\limits^{\bf L -1}_{\bf k=0}
D_{\bf L}^{\bf k}\otimes  A_{\bf k}) (F_{\bf L} \otimes I_{m_0}).
\]
For large $L$ the numerical cost becomes prohibitive. 
However, the above representation indicates that the further storage and complexity reduction 
can be possible if
the third-order coefficients tensor ${\bf A}= [A_{k_1 k_2 k_3}]$, $k_\ell=0,...,L_\ell-1$,  
with the matrix-valued entries $A_{k_1 k_2 k_3}\in \mathbb{R}^{m_0\times m_0}$, 
allows the low-rank tensor factorization (approximation) in the multiindex ${\bf k}=(k_1, k_2, k_3)$,
which can be described by a smaller then $L^3$ number of parameters.


To fix the idea, let us assume the existence of rank-$1$ separable tensor factorization,
\begin{equation} \label{eqn:SepMatrBlock}
 A_{k_1 k_2 k_3} = A_{k_1}^{(1)}\odot A_{k_2}^{(2)} \odot A_{k_3}^{(3)},
 \quad A_{k_1}^{(1)}, A_{k_2}^{(2)},A_{k_3}^{(3)} \in \mathbb{R}^{m_0\times m_0},
\quad \mbox{for} \quad k_\ell=0,...,L_\ell-1.
\end{equation}

Given $\ell \in \{1,...,d\}$ and a matrix $G\in \mathbb{R}^{L_\ell \times L_\ell}$, define 
the {\it tensor prolongation} (lifting) mapping, 
${\cal P}_\ell: \mathbb{R}^{L_\ell\times L_\ell}\to \mathbb{R}^{|{\bf L}|\times |{\bf L}|}$, by
\begin{equation} \label{eqn:Tensor_prolong}
{\cal P}_\ell(G):= \left(\bigotimes_{i=1}^{\ell-1}I_{L_i}\right)
\otimes G \otimes \left(\bigotimes_{i=\ell+1}^{d}I_{L_i}\right).
\end{equation}
The following theorem introduces the new multilevel block-circulant tensor-structured
matrix format, 
where the coefficient tensor  ${\bf A}$ is represented via the low-rank factorization.
\begin{theorem}\label{thm:tens_FFT}
Assume the separability of a tensor $[A_{\bf k}]$ in  the ${\bf k}$ space in the 
form (\ref{eqn:SepMatrBlock}), 
then the $3$-level block-circulant matrix $A$ can be represented in the factorized block-diagonal 
form as follows 
\begin{equation} \label{eqn:Tensor_Circl_form}
A= (F_{\bf L}^\ast \otimes I_{m_0}) D_A (F_{\bf L}\otimes I_{m_0}),
\end{equation}
where the block-diagonal matrix $D_A$ with the block size $ m_0 \times m_0$ is given by
\[
 D_A = {\cal P}_1(\mathrm{bdiag}F_{L_1} {\bf A}^{(1)}) \odot
  {\cal P}_2(\mathrm{bdiag}F_{L_2} {\bf A}^{(2)})\odot 
  {\cal P}_3 (\mathrm{bdiag}F_{L_3} {\bf A}^{(3)}),
\]
with tri-tensors ${\bf A}^{(\ell)}=[A_{0}^{(\ell)},...,A_{L_\ell-1}^{(\ell)}]^T
\in \mathbb{R}^{L_\ell\times m_0 \times m_0}$ defined by
 concatenation of $\ell$-factors in (\ref{eqn:SepMatrBlock}).
\end{theorem}
\begin{proof}
The diagonal blocks in (\ref{eqn:MLbcircDiag2}) can be written in 
the factorized tensor-product form
\begin{align*}\label{eqn:bcirc_R1}
& D_{L_1}^{k_1} \otimes D_{L_2}^{k_2}\otimes D_{L_3}^{k_3}\otimes A_{k_1 k_2 k_3} =\\
 = &
((D_{L_1}^{k_1}\otimes A_{k_1}^{(1)}) \otimes I_{L_2}\otimes I_{L_3}  )\odot
 (I_{L_1} \otimes (D_{L_2}^{k_2} \otimes A_{k_2}^{(2)})\otimes I_{L_3} ) \odot
 (I_{L_1} \otimes I_{L_2}\otimes (D_{L_3}^{k_3} \otimes A_{k_3}^{(3)})). 
\end{align*}
Combining this representation with (\ref{eqn:MLbcircDiag2})
leads to the powerful matrix factorization
\begin{align*}
 A=& (F_{\bf L}^\ast \otimes I_{m_0}) 
 \left[\sum\limits^{L_1 -1}_{k_1=0}  {\cal P}_1(D_{L_1}^{k_1}\otimes A_{k_1}^{(1)}) \odot
 \sum\limits^{L_2 -1}_{k_2=0}   {\cal P}_2(D_{L_2}^{k_2} \otimes A_{k_2}^{(2)}) \odot
 \sum\limits^{L_3 -1}_{k_3=0}  {\cal P}_3(D_{L_3}^{k_3} \otimes A_{k_3}^{(3)})\right]
 (F_{\bf L}\otimes I_{m_0})\\
=& (F_{\bf L}^\ast \otimes I_{m_0}) 
 \left[ {\cal P}_1(\sum\limits^{L_1-1}_{k_1=0} D_{L_1}^{k_1}\otimes A_{k_1}^{(1)}) \odot
    {\cal P}_2(\sum\limits^{L_2 -1}_{k_2=0} D_{L_2}^{k_2} \otimes A_{k_2}^{(2)}) \odot
   {\cal P}_3(\sum\limits^{L_3 -1}_{k_3=0} D_{L_3}^{k_3} \otimes A_{k_3}^{(3)})\right]
 (F_{\bf L}\otimes I_{m_0})\\
 = & (F_{\bf L}^\ast \otimes I_{m_0})
 \left[ {\cal P}_1(\mbox{bdiag} F_{L_1} {\bf A}^{(1)}) \odot
  {\cal P}_2(\mbox{bdiag} F_{L_2} {\bf A}^{(2)})\odot
  {\cal P}_3(\mbox{bdiag} F_{L_3} \otimes {\bf A}^{(3)})\right]
 (F_{\bf L}\otimes I_{m_0}),
\end{align*}
where the tensor prolongation operator ${\cal P}_\ell$ is given by (\ref{eqn:Tensor_prolong}).
\end{proof}

The expansion (\ref{eqn:Tensor_Circl_form}) includes only 1D Fourier transforms 
thus reducing the representation cost to 
$$
O(m_0^2 {\sum}_{\ell=1}^d L_\ell \log L_\ell).
$$ 
Moreover, and it is even more important, that the eigenvalue 
problem  for the large matrix $A$ now reduces to only 
 $L_1+L_2+L_3 \ll L_1 L_2 L_3$ independent small $m_0 \times m_0$ matrix eigenvalue problems.

The above block-diagonal representation for $d=3$  
generalizes easily to the case of arbitrary dimension $d$.
Furthermore,
the rank-$1$ decomposition (\ref{eqn:SepMatrBlock}) was considered for the ease of exposition only.
For instance, the above low-rank representations can be easily generalized to the case of 
canonical (CP) or Tucker formats in ${\bf k}$ space 
(see Proposition \ref{prop:low_rank_coef} below). In fact, both CP and Tucker formats
provide the additive structure which can be converted to the respective additive 
structure of the core coefficient in (\ref{eqn:Tensor_Circl_form}).

Notice that in the practically interesting 3D case
the use of MPS/TT type factorizations does not take the advantage over the Tucker format
since the Tucker and MPS ranks in 3D  appear to be close to each other.
Indeed, the HOSVD for a tensor of order $3$ leads to the same sharp rank estimates for both
the Tucker and TT tensor formats.

\subsection{Block circulant structure in the periodic core Hamiltonian}
\label{ssec:Core_Ham_period}

In this section we consider the periodic case, further called case (P), and
derive the more refined sparsity pattern of the matrix $V_{c_L}$ in (\ref{nuc_potMatrTot})
by using the $d$-level ($d=1,2,3$) tensor structure  in this matrix.
The matrix block entries are numbered by a pair of multi-indices, 
$V_{c_L}=\{V_{{\bf k}{\bf m}}\}$, ${\bf k}=(k_1,k_2,k_3)$, where the $m_0\times m_0$ 
matrix block $V_{{\bf k}{\bf m}}$ is defined by (\ref{eqn:nuc_MatrSparsP}).
Figure \ref{fig:cube_box}  illustrates an example of 3D lattice-type structure of 
size $4\times 4\times 2$. 
  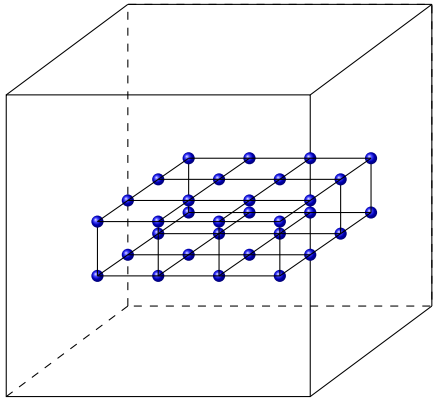
\begin{figure}[htb]
\centering
\begin{tikzpicture}[scale=0.4]

\draw (2,2) rectangle (12,12);
\path[draw] (2,12) -- (6,15);
\path[draw] (12,12) -- (16,15);

\draw[style=dashed] (6,5) rectangle (16,15);
\path[draw][style=dashed] (2,2) -- (6,5);
\path[draw] (12,2) -- (16,5);
\path[draw] (6,15) -- (16,15);
\path[draw] (16,5) -- (16,15);
--------------------------------------

\foreach \x in {5,7,9,11} {
\shade[shading = ball, ball color = blue] (\x,6) circle (.2);
\shade[shading = ball, ball color = blue] (\x+1,6.7) circle (.2);
\shade[shading = ball, ball color = blue] (\x+2,7.4) circle (.2);
\shade[shading = ball, ball color = blue] (\x+3,8.1) circle (.2);
}
  \path[draw] (5,6) -- (11,6);

\foreach \x in {5,7,9,11} {
\path[draw] (\x,6) -- (\x+3,8.1);
}
\path[draw] (6,6.7) -- (12,6.7);
\path[draw] (7,7.4) -- (13,7.4);
\path[draw] (8,8.1) -- (14,8.1);

\foreach \x in {5,7,9,11} {
\path[draw] (\x,6) -- (\x,7.8);
}
 \path[draw] (14,8.1) -- (14,9.9);
 \path[draw] (13,7.4) -- (13,9.2);
 \path[draw] (8,8.1) -- (8,9.9);

\foreach \x in {5,7,9,11} {
\shade[shading = ball, ball color = blue] (\x,7.8) circle (.2);
\shade[shading = ball, ball color = blue] (\x+1,8.5) circle (.2);
\shade[shading = ball, ball color = blue] (\x+2,9.2) circle (.2);
\shade[shading = ball, ball color = blue] (\x+3,9.9) circle (.2);
}
\path[draw] (5,7.8) -- (11,7.8);

\foreach \x in {5,7,9,11} {
\path[draw] (\x,7.8) -- (\x+3,9.9);
}
\path[draw] (6,8.5) -- (12,8.5);
\path[draw] (7,9.2) -- (13,9.2);
\path[draw] (8,9.9) -- (14,9.9);

\end{tikzpicture}
\caption{Rectangular  $4\times 4\times 2$ lattice in a box.}
\label{fig:cube_box}
\end{figure}

Following \cite{VeBoKh:Ewald:14}, we introduce the periodic cells 
${\cal R}=  \mathbb{Z}^d$, $d=1,2,3$ for the $\bf k$ index, and consider a 3D $B$-periodic 
supercell $\Omega_L= B\times B\times B$, with $B= \frac{b}{2}[-L,L]$. 
The total electrostatic potential in the supercell $\Omega_L$ is obtained by, first,  
the lattice summation of the Coulomb potentials over  $\Omega_L$ for (rather large) $L$, 
but restricted to the central unit cell $\Omega_0$, and then by replication of the resultant
function to the whole supercell $\Omega_L$. 
Hence, in this construction, the total potential sum $v_{c_L}(x)$ is designated at each elementary 
unit-cell in $\Omega_L$ by the same value (${\bf k}$-translation invariant).
The electrostatic potential in each of $B$-periods can be obtained by copying the 
respective data from $\Omega_L$. 

The effect of the conditional convergence of
the lattice summation as $L\to \infty$ can be treated by using the  extrapolation 
to the limit (regularization) on 
a sequence of different lattice parameters $L, 2L, 3L, \ldots$ as described in \cite{VeBoKh:Ewald:14}.

Consider the case $d=3$ in the more detail.  
Recall that the reference value $v_{c_L}(x)$ will be computed at the central cell
$\Omega_0$, indexed by $(0,0,0)$, by summation over all contributions from $L^3$ 
elementary sub-cells in $\Omega_L$. For technical reasons here and in the following we
vary the summation index by $k_\ell=0,..., L-1$, to obtain
\begin{equation}\label{eqn:EwaldSumP}
 v_0(x)=   \sum\limits_{k_1,k_2,k_3=0}^{L-1} \sum_{\nu=1}^{M_0}
 \frac{Z_\nu}{\|{x} -a_\nu (k_1,k_2,k_3)\|}, \quad x\in \Omega_0.
\end{equation}

In the following, we use the same notations as in \S2.3.
The basis set in $\Omega_L$ is constructed by replication 
from the master unit cell $\Omega_0$ to the whole periodic lattice.
The tensor representation of the local lattice sum on the $n\times n \times n$ grid 
associated with $\Omega_0$ takes a form
\[
 {\bf P}_{\Omega_0} 
= \sum_{\nu=1}^{M_0} Z_\nu \sum\limits_{k_1,k_2,k_3=0}^{L-1} \sum\limits_{r=1}^{R}
{\cal W}_{\nu({\bf k})} \widetilde{\bf p}^{(1)}_{r} 
\otimes \widetilde{\bf p}^{(2)}_{r} \otimes \widetilde{\bf p}^{(3)}_{r}
\in \mathbb{R}^{n\times n \times n},
\]
where the tensor ${\bf P}_{\Omega_0}$ of size $n\times n \times n$ allows the low-rank
expansion as in (\ref{eqn:EwaldTensorGl}) with the reference tensor 
$\widetilde{\bf P}_{R}$  defined by (\ref{eqn:master_pot}).
Here,  the $\Omega$-windowing operator,
$
{\cal W}_{\nu({\bf k})}={\cal W}_{\nu(k_1)}^{(1)}\otimes {\cal W}_{\nu(k_2)}^{(2)}
\otimes {\cal W}_{\nu(k_3)}^{(3)},
$
restricts onto the $n\times n \times n$ unit cell by shifting via the lattice vector 
${\bf k}=(k_1,k_2,k_3)$. This reduces both the computational and storage costs by a factor of $L$.

In the 3D case, we set $q=3$ in the notation for multilevel block-circulant (BC) matrix,
see Appendix.
Similar to the case of one-level BC matrices, we notice that a matrix 
$A\in {\cal BC} (3,{\bf L},m)$ of size $|{\bf L}| m \times  |{\bf L}| m$ is completely 
defined  by a $3$-rd order coefficients tensor ${\bf A}=[A_{k_1 k_2 k_3}]$ of size 
$L_1 \times L_2 \times L_3 $,
($k_\ell=0,...,L_\ell-1$, $\ell=1,2,3$) with $m\times m$ block-matrix entries,  
obtained by folding of the generating first column vector in $A$. 

\begin{lemma}\label{lem:SparseCaseP}
Assume that in case (P) the number of overlapping unit cells 
(in the sense of effective supports of basis functions) in each spatial direction 
does not exceed $L_0$. 
Then the Galerkin matrix $V_{c_L}=[V_{{\bf k}{\bf m}}]$ exhibits the symmetric, three-level block circulant 
Kronecker tensor-product form, i.e. $V_{c_L} \in {\cal BC} (3,{\bf L},m_0)$,
\begin{equation}\label{eqn:BC-Core}
V_{c_L}= \sum\limits_{k_1=0}^{L_1-1} \sum\limits_{k_2=0}^{L_2-1} \sum\limits_{k_3=0}^{L_3-1}
\pi_{L_1}^{k_1}\otimes \pi_{L_2}^{k_2}\otimes \pi_{L_3}^{k_3}\otimes A_{k_1 k_2 k_3}, 
\quad A_{k_1 k_2 k_3}\in \mathbb{R}^{m_0\times m_0},
\end{equation}
where the number of non-zero matrix blocks $A_{k_1 k_2 k_3}$ does not exceed $(L_0+1)^3$.
Similar properties hold for both the Laplacian and the mass matrix.

The required storage is bounded by $m_0^2 (L_0 + 1)^3$ independent of $L$.
The set of non-zero generating matrix blocks $\{ A_{k_1 k_2 k_3}\}$ can be calculated 
in $O(m_0^2 (L_0 + 1)^3 n)$ operations.

Furthermore, assume that the QTT ranks of the assembled canonical vectors do not exceed $r_0$. 
Then the numerical cost can be reduced to the logarithmic scale, $O( m_0^2 (L_0 + 1)^3 r_0^2 \log n)$.
\end{lemma}
\begin{proof}
First, we notice that the shift invariance property in the matrix 
$V_{c_L}=[V_{{\bf k}{\bf m}}]$  is a consequence of the
translation invariance in the canonical tensor ${\bf P}_{c_L}$ (periodic case), 
and in the tensor ${\bf G}_{\bf k}$ representing basis functions (by construction),
\begin{equation} \label{eqn:Basis_shift}
{\bf G}_{\bf k m}:= {\bf G}_{\bf k} \odot {\bf G}_{\bf m}= {\bf G}_{|{\bf k} -{\bf m}|} \quad 
\mbox{for} \quad 0\leq k_\ell, m_\ell \leq L-1, 
\end{equation}
so that we have
\begin{equation} \label{nuc_BCirculantP}
V_{{\bf k}{\bf m}} = V_{|{\bf k}-{\bf m}|}, \quad 0\leq k_\ell, m_\ell \leq L-1.
\end{equation}
This ensures the perfect three-level block-Toeplitz structure of $V_{c_L}$ 
(compare with the case of a bounded box). 
Now the block circulant pattern characterizing the class ${\cal BC} (3,{\bf L},m_0)$ 
is imposed by the periodicity of a lattice-structured basis set.

To prove the complexity bounds we observe that a matrix 
$V_{c_L} \in {\cal BC} (3,{\bf L},m_0)$
can be represented in the Kronecker tensor product form (\ref{eqn:BC-Core}),
obtained by an easy generalization of (\ref{eqn:bcircPol}). In fact, we apply 
(\ref{eqn:bcircPol}) by successive slice-wise and fiber-wise unfolding
to obtain
\[
\begin{split}
 V_{c_L}
& = \sum\limits_{k_1=0}^{L_1-1}\pi_{L_1}^{k_1}\otimes {\bf A}_{k_1} \\
&= \sum\limits_{k_1=0}^{L_1-1}\pi_{L_1}^{k_1}\otimes 
\left( \sum\limits_{n_2=0}^{L_2-1} \pi_{L_2}^{k_2}\otimes {\bf A}_{k_1 k_2} \right)\\
&=\sum\limits_{k_1=0}^{L_1-1}\pi_{L_1}^{k_1}\otimes 
\left( \sum\limits_{k_2=0}^{L_2-1} \pi_{L_2}^{k_2}\otimes 
\left(\sum\limits_{k_3=0}^{L_3-1} \pi_{L_3}^{k_3}\otimes {A}_{k_1 k_2 k_3} \right) \right),
\end{split}
\]
where ${\bf A}_{k_1}\in \mathbb{R}^{L_2\times L_3 \times m_0\times m_0}$, 
${\bf A}_{k_1 k_2} \in  \mathbb{R}^{L_3 \times m_0\times m_0}$, and 
$A_{k_1 k_2 k_3}\in \mathbb{R}^{m_0\times m_0}$.
Now the overlapping assumption ensures that the number of non-zero matrix blocks 
$A_{k_1 k_2 k_3}$ does exceed $(L_0+1)^3$.

Furthermore, the symmetric mass matrix, $S_{c_L}=\{{s}_{\mu \nu} \}\in \mathbb{R}^{N_b\times N_b}$, 
providing the Galerkin  representation of the identity operator reads as follows, 
\begin{equation*}  \label{Ident_pot}
 {s}_{\mu \nu}=  
 \langle {\bf G}_\mu , {\bf G}_\nu \rangle 
 =\langle S^{(1)}{\bf g}_\mu^{(1)},{\bf g}_\nu^{(1)} \rangle 
 \langle S^{(2)} {\bf g}_\mu^{(2)},{\bf g}_\nu^{(2)} \rangle 
 \langle S^{(3)} {\bf g}_\mu^{(3)},{\bf g}_\nu^{(3)} \rangle, 
\quad 1\leq \mu, \nu \leq N_b,
\end{equation*}
 where $N_b=m_0 L^3$.
It can be seen that in the periodic case the block structure in the "basis-tensor"  
${\bf G}_{\bf k} $ imposes
 the three-level block circulant structure in the mass matrix $S_{c_L}$, 
\begin{equation}\label{eqn:BC-mass}
S_{c_L}= \sum\limits_{k_1=0}^{L_1-1} \sum\limits_{k_2=0}^{L_2-1} \sum\limits_{k_3=0}^{L_3-1}
\pi_{L_1}^{k_1}\otimes \pi_{L_2}^{k_2}\otimes \pi_{L_3}^{k_3}\otimes S_{k_1 k_2 k_3}, 
\quad S_{k_1 k_2 k_3}\in \mathbb{R}^{m_0\times m_0}.
\end{equation}
By the previous arguments we conclude that 
$S_{k_1 k_2 k_3}=S^{(1)}_{k_1} S^{(2)}_{k_2} S^{(3)}_{k_3}$ implying the rank-$1$
separable representation in (\ref{eqn:BC-mass}).

Likewise, it is easy to see that the stiffness matrix representing the (local) Laplace 
operator in the periodic setting has the similar block circulant structure, 
\begin{equation}\label{eqn:BC-Laplace}
\Delta_{c_L}= \sum\limits_{k_1=0}^{L_1-1} \sum\limits_{k_2=0}^{L_2-1} \sum\limits_{k_3=0}^{L_3-1}
\pi_{L_1}^{k_1}\otimes \pi_{L_2}^{k_2}\otimes \pi_{L_3}^{k_3}\otimes B_{k_1 k_2 k_3}, 
\quad B_{k_1 k_2 k_3}\in \mathbb{R}^{m_0\times m_0},
\end{equation}
where the number of non-zero matrix blocks $B_{k_1 k_2 k_3}$ does not exceed $(L_0+1)^3$.
In this case the matrix block $B_{k_1 k_2 k_3}$ admits a rank-$3$ product factorization
inheriting the tri-term representation of the Laplacian. 
\end{proof}

In the Hartree-Fock calculations for lattice structured systems we deal with 
the multilevel, symmetric block circulant/Toeplitz matrices, where
the first-level blocks, $A_0,...,A_{L_1-1}$, may have further block structures.
In particular, Lemma \ref{lem:SparseCaseP} shows that 
the Galerkin approximation of the 3D Hartree-Fock core Hamiltonian $H$ in periodic setting 
leads to the symmetric, three-level block circulant matrix structures.

\begin{figure}[htbp]
\centering
\includegraphics[width=6.0cm]{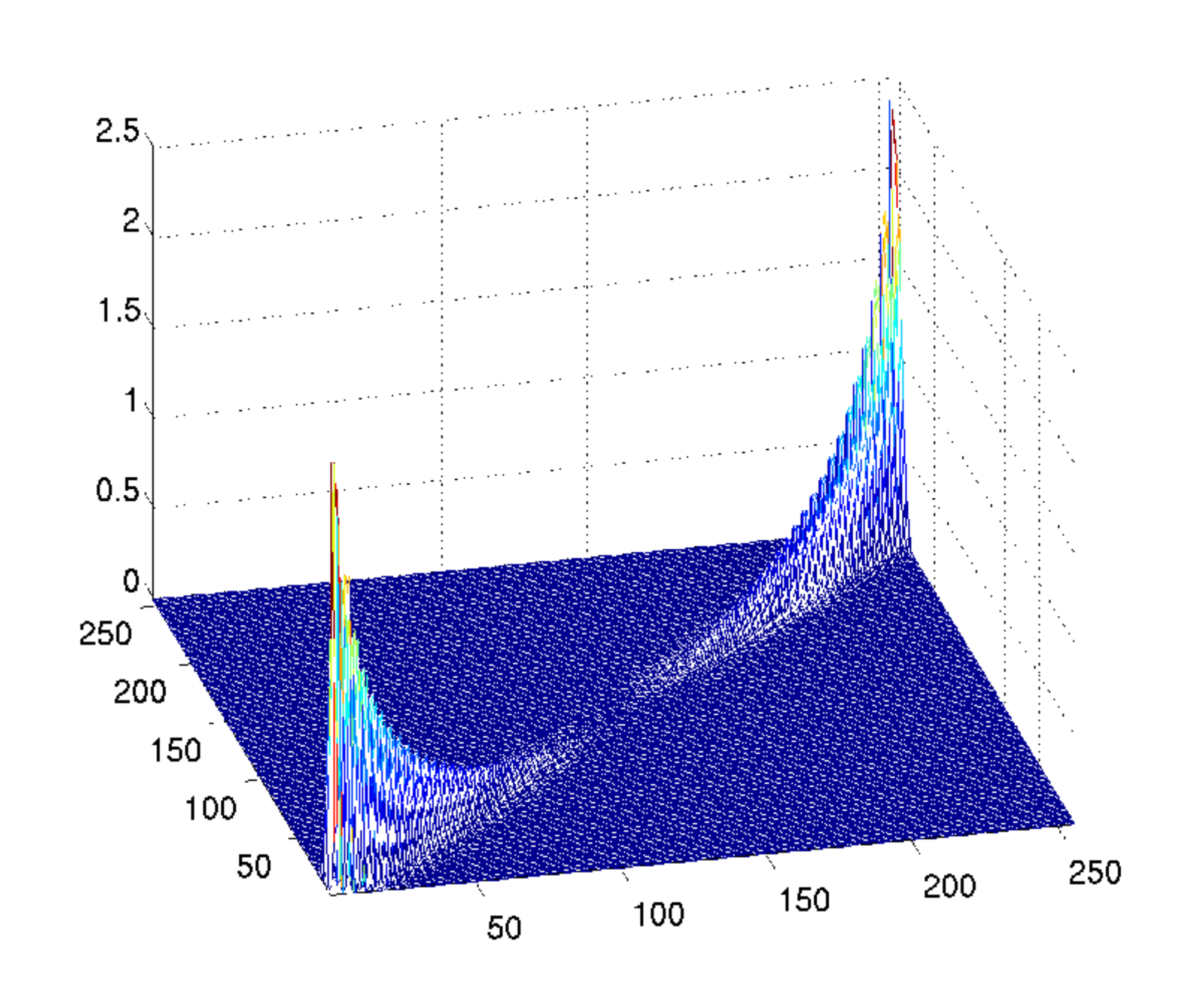} 
\caption{\small  
Difference between  matrices $V_{c_L}$ in periodic and non-periodic cases, $L=64$.
}
\label{fig:3DCorePerScellErr}  
\end{figure}
Figure \ref{fig:3DCorePerScellErr} shows the difference between matrices $V_{c_L}$ in periodic 
and non-periodic cases. 

Figure \ref{fig:3DCoreHamPer} represents the block-sparsity in the core Hamiltonian 
matrix of a $L \times 1\times 1$ Hydrogen chain  in a box with $L=32$ (right), 
and  the matrix profile (left).

\begin{figure}[htbp]
\centering
\includegraphics[width=6.0cm]{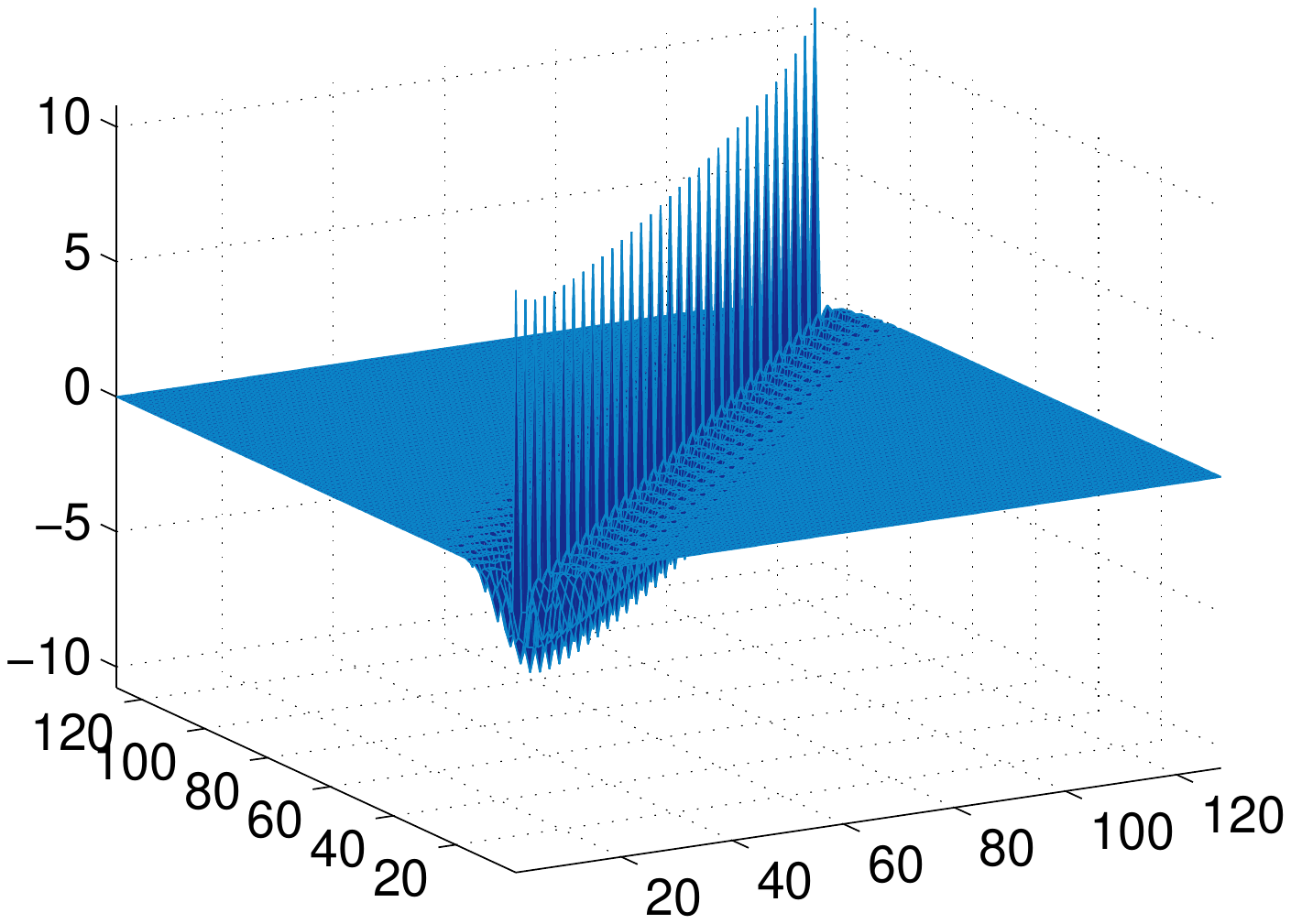}\quad \quad
\includegraphics[width=6.0cm]{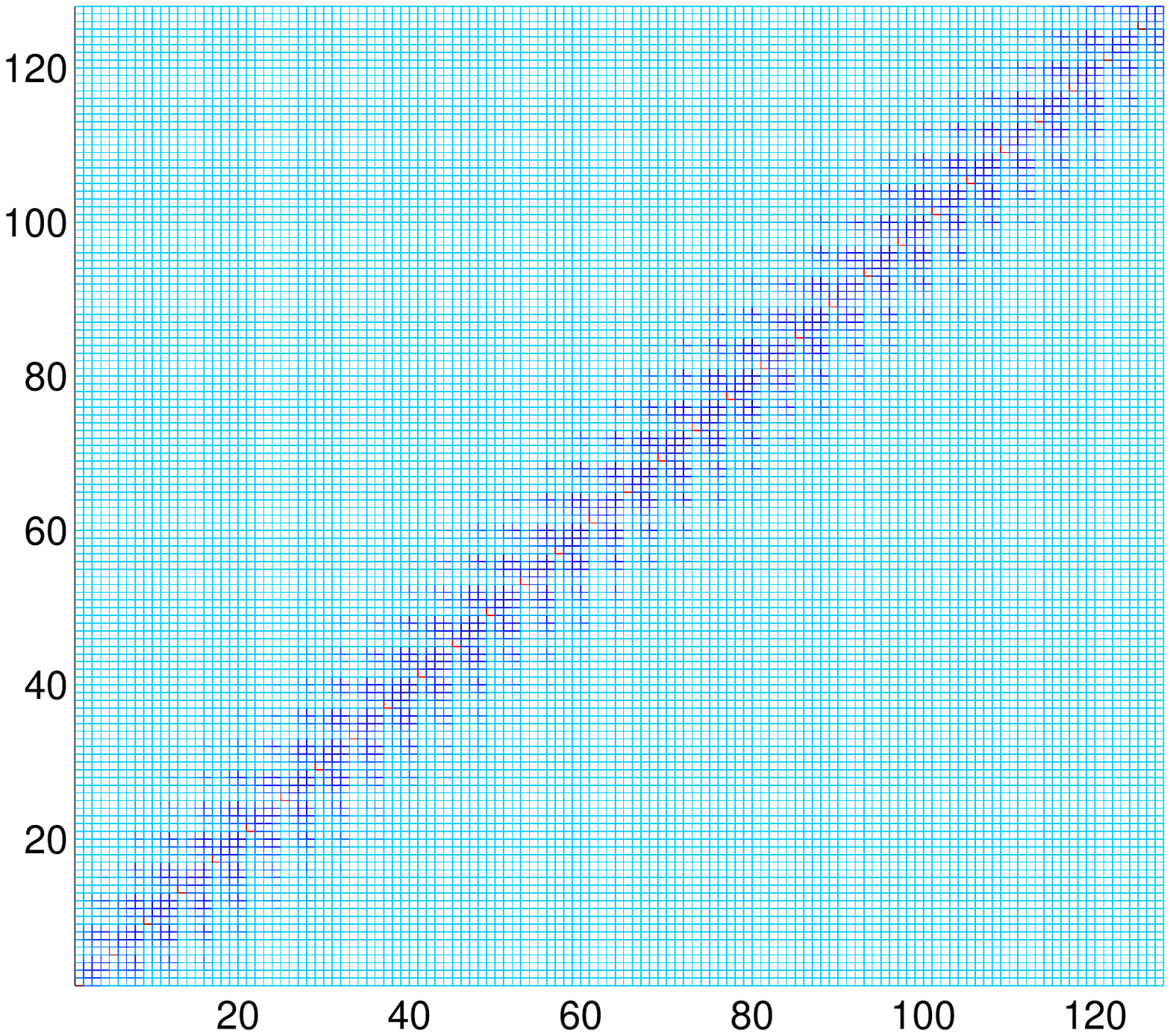}
\caption{\small  Block-sparsity in the matrix $V_{c_L}$  in a box for $L=32$ (right); 
matrix profile (left).}
\label{fig:3DCoreHamPer}  
\end{figure}

In the next section we discuss computational details of the FFT-based eigenvalue solver on
the example of 3D linear chain of molecules. 

\subsection{Spectral problems in different settings: complexity analysis}\label{ssec:Complexity_EigPr}

Combining the block circulant representations (\ref{eqn:BC-Core}), (\ref{eqn:BC-Laplace}) 
and (\ref{eqn:BC-mass}),
we are able to represent the eigenvalue problem for the Fock matrix 
in the Fourier space as follows
\begin{equation}\label{eqn:HF-FSpace}
 (\Delta_{c_L} + V_{c_L})U = \lambda S_{c_L} U, 
\end{equation}
where $U=(F_{\bf L}\otimes I_m) C$ and
\begin{equation*}\label{eqn:HF-FSpace_Repres}
\Delta_{c_L} + V_{c_L}=
 \sum\limits^{\bf L}_{{\bf k}=0}  D_{L_1}^{k_1}\otimes D_{L_2}^{k_2}\otimes D_{L_3}^{k_3}\otimes 
      (B_{k_1 k_2 k_3} + A_{k_1 k_2 k_3}), \quad
 S_{c_L} = 
\sum\limits^{\bf L}_{{\bf k}=0}
D_{L_1}^{k_1}\otimes D_{L_2}^{k_2}\otimes D_{L_3}^{k_3} S_{k_1 k_2 k_3}, 
\end{equation*}
with the diagonal matrices $D_{L_\ell}^{k_\ell}\in \mathbb{R}^{L_\ell \times L_\ell}$, 
$\ell=1,2,3$ and the compact notation 
$$
\sum\limits^{\bf L}_{{\bf k}=0} = 
\sum\limits^{L_1 -1}_{k_1=0}\sum\limits^{L_2 -1}_{k_2=0}\sum\limits^{L_3 -1}_{k_3=0}.
$$
The equivalent block-diagonal form reads 
\begin{equation}\label{eqn:HF-FSpace-Block}
 \mbox{bdiag}_{m_0\times m_0} 
 \{{\cal T}_{\bf L}'[F_{\bf L} ({\cal T}_{\bf L}\widehat{B}) 
+ F_{\bf L} ({\cal T}_{\bf L}\widehat{A})] -
\lambda {\cal T}_{\bf L}'(F_{\bf L} [{\cal T}_{\bf L} \widehat{S})]  \} U =0.
\end{equation}
The block structure specified by Lemma \ref{lem:SparseCaseP} allows to apply the efficient
eigenvalue solvers via FFT based diagonalization in the framework of Hartree-Fock calculations
 with the numerical cost $O(m_0^2 L^d \log L)$.
\begin{figure}[htbp]
\centering
\includegraphics[width=6.0cm]{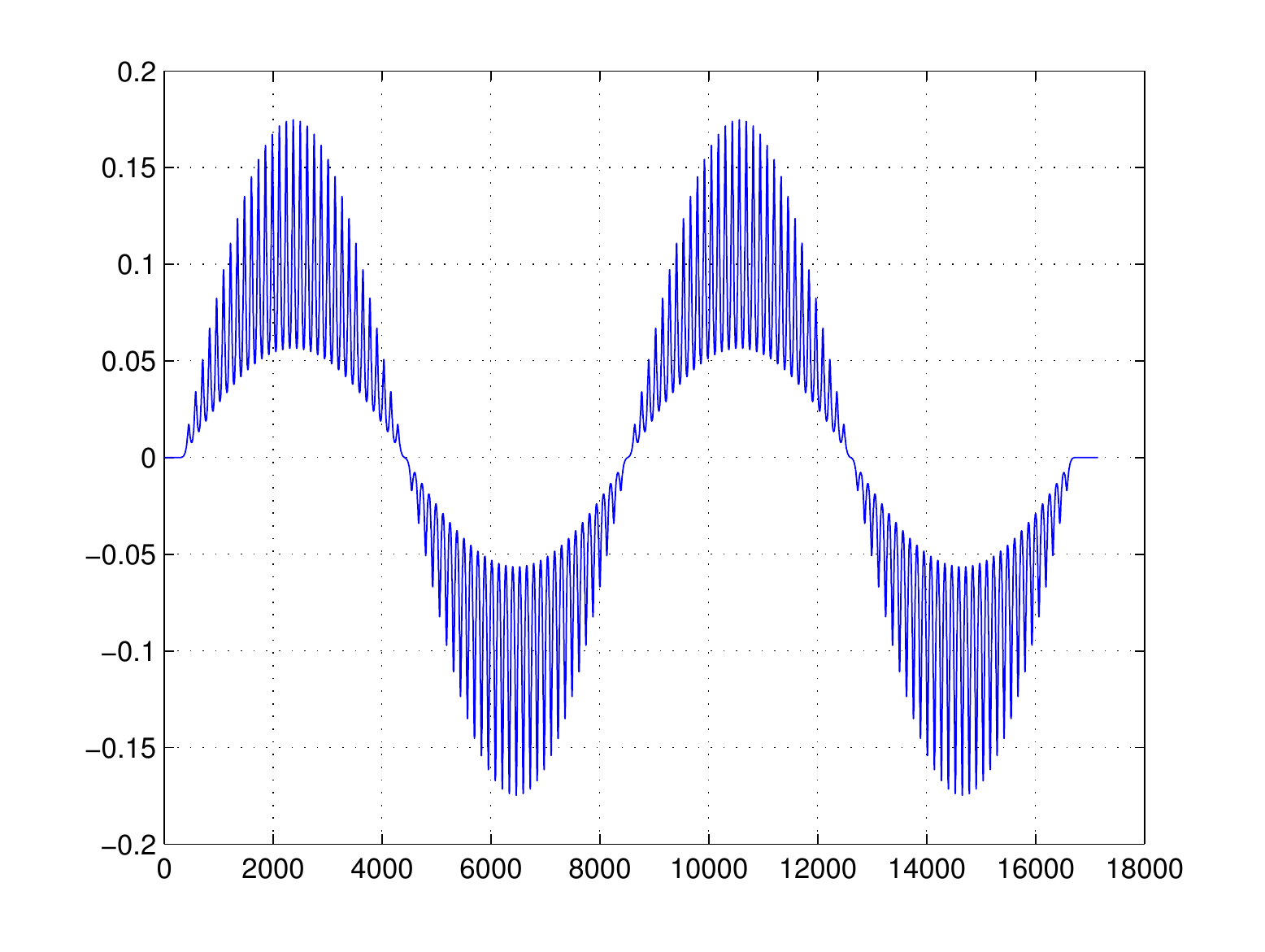} \quad
\includegraphics[width=6.0cm]{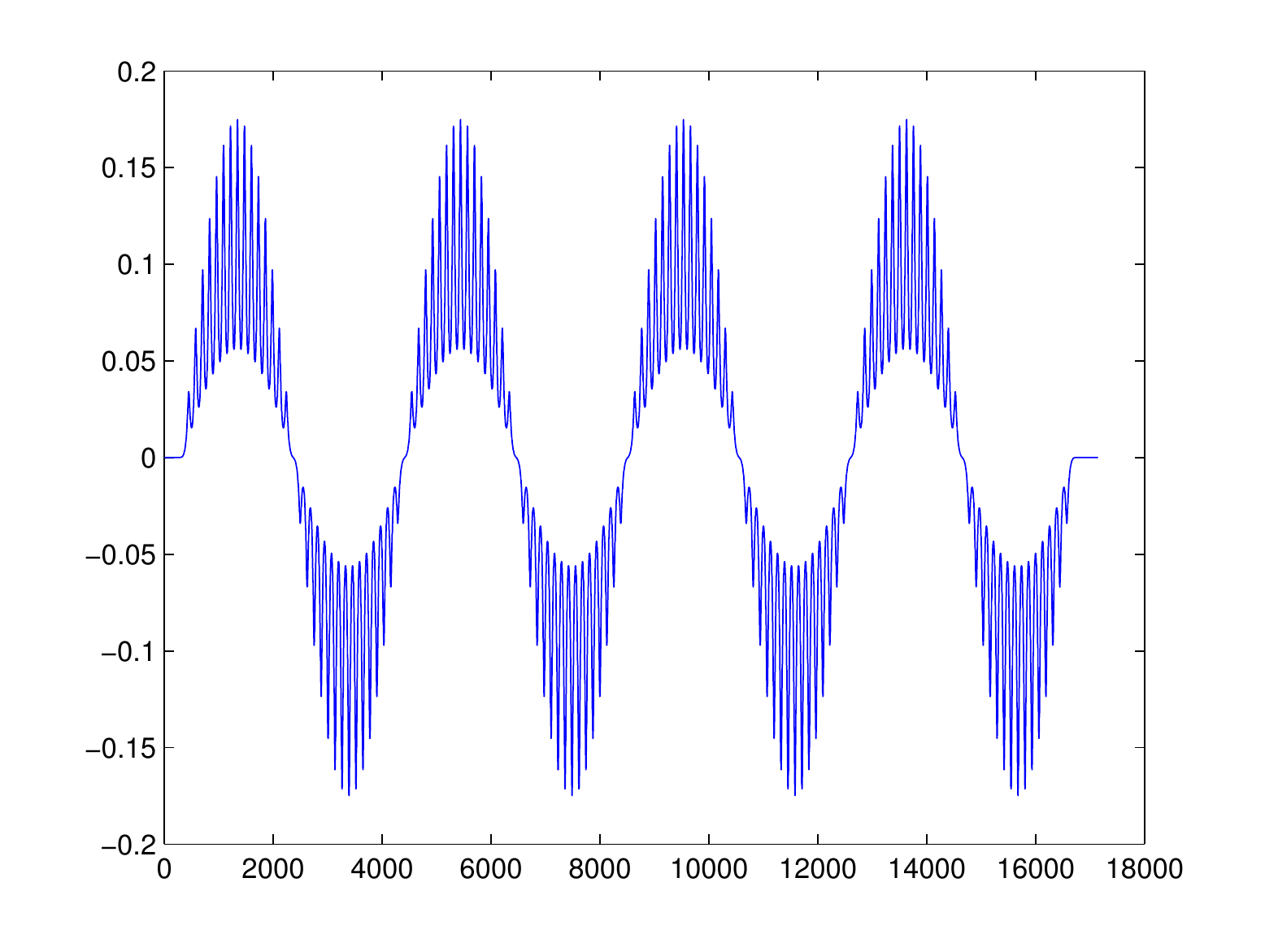}
\caption{\small  Molecular orbitals, i.e. the eigenvectors represented in GTO basis: 
the $4$th orbital (left), the $8$th orbital (right). 
}
\label{fig:3DCoreEigVect}  
\end{figure}
Figure \ref{fig:3DCoreEigVect} visualizes molecular orbitals on fine spatial grid 
with $n=2^{14}$: the $4$th orbital (left), the $8$th orbital (right).
The eigenvectors are computed in GTO basis for $L\times 1\times 1$ system with $L=128$ and $m_0=4$.

\begin{remark}\label{prop:low_rank_coef}
The low-rank structure in the coefficients tensor mentioned above 
(see Section \ref{ssec:Tensor_bcirc}) allows to reduce the factor $L^d \log L$ to 
$L \log L$ for $d=2,3$. It was already observed in the proof of Lemma \ref{lem:SparseCaseP}
that the respective coefficients in the overlap and Laplacian Galerkin matrices
can be treated as the rank-$1$ and rank-$3$ tensors, respectively.
Clearly, the factorization rank for the nuclear part of the Hamiltonian does not exceed 
$R$. Hence, Theorem \ref{thm:tens_FFT}  can be applied in the generalized form.
\end{remark}

\begin{table}[htb]
\begin{center}%
\begin{tabular}
[c]{|r|r|r|r|r|r|r|r|r|r|}%
\hline
Matrix size $N_b=m_0 L $      & $512$  & $1024$ & $2048$ & $4096$ & $8192$ & $16384$ & $32768$ & $65536$ & $131072$ \\ 
 \hline 
Full EIG-solver         &  $0.67$    &    $5.49$  &   $48.6 $   &  $497.4$    & $--$ & $--$ & $--$ & $--$ & $--$\\
 \hline 
MBC diagonalization  &  $0.10$ &   $0.09$  &   $0.08 $   &  $0.14$    & $0.44$ & $1.5$ & $5.6$ & $22.9$ & $89.4$ \\
 \hline 
 \end{tabular}
\caption{CPU times (sec.): full eig-solver  vs. FFT-based MBC diagonalization for 
 $L\times 1\times 1$ lattice system, and with $m_0=4$, $L=2^p$, $p=7,8,...,15$. }
\label{Table_Times_SupSvsPer}
\end{center}
\end{table}
Table \ref{Table_Times_SupSvsPer} compares CPU times in sec. (Matlab) for the 
full eigenvalue solver on a 3D $L\times 1\times 1$ lattice in a box, and for
the FTT-based MBC diagonalization in the periodic supercell, 
all computed for $m_0=4$, $L=2^p$ ($p=7,8,...,15$).
The number of basis functions (problem size) is given by $N_b=m_0 L$.

\begin{figure}[htb]
\centering
\includegraphics[width=6.0cm]{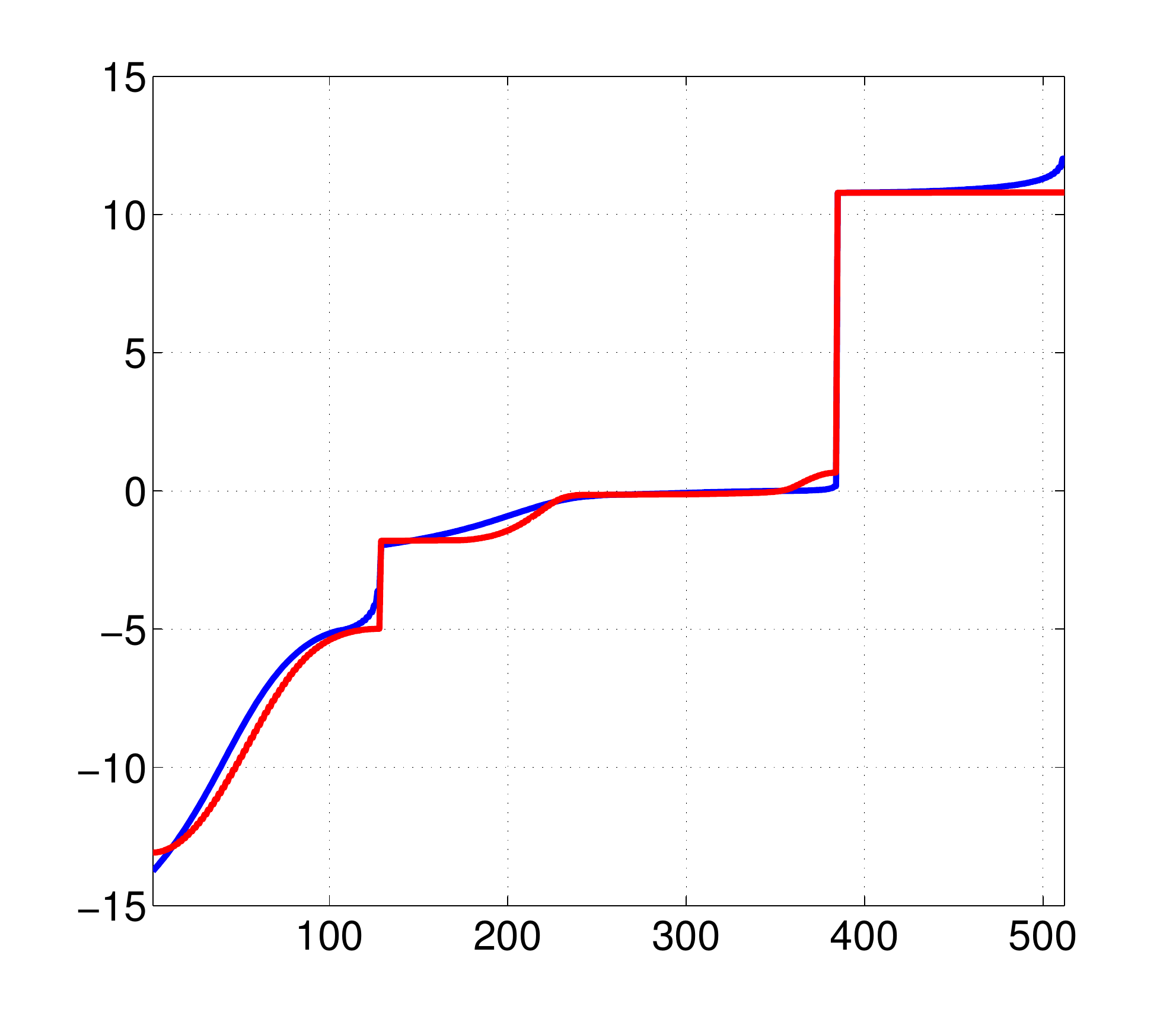}\quad  
\includegraphics[width=6.0cm]{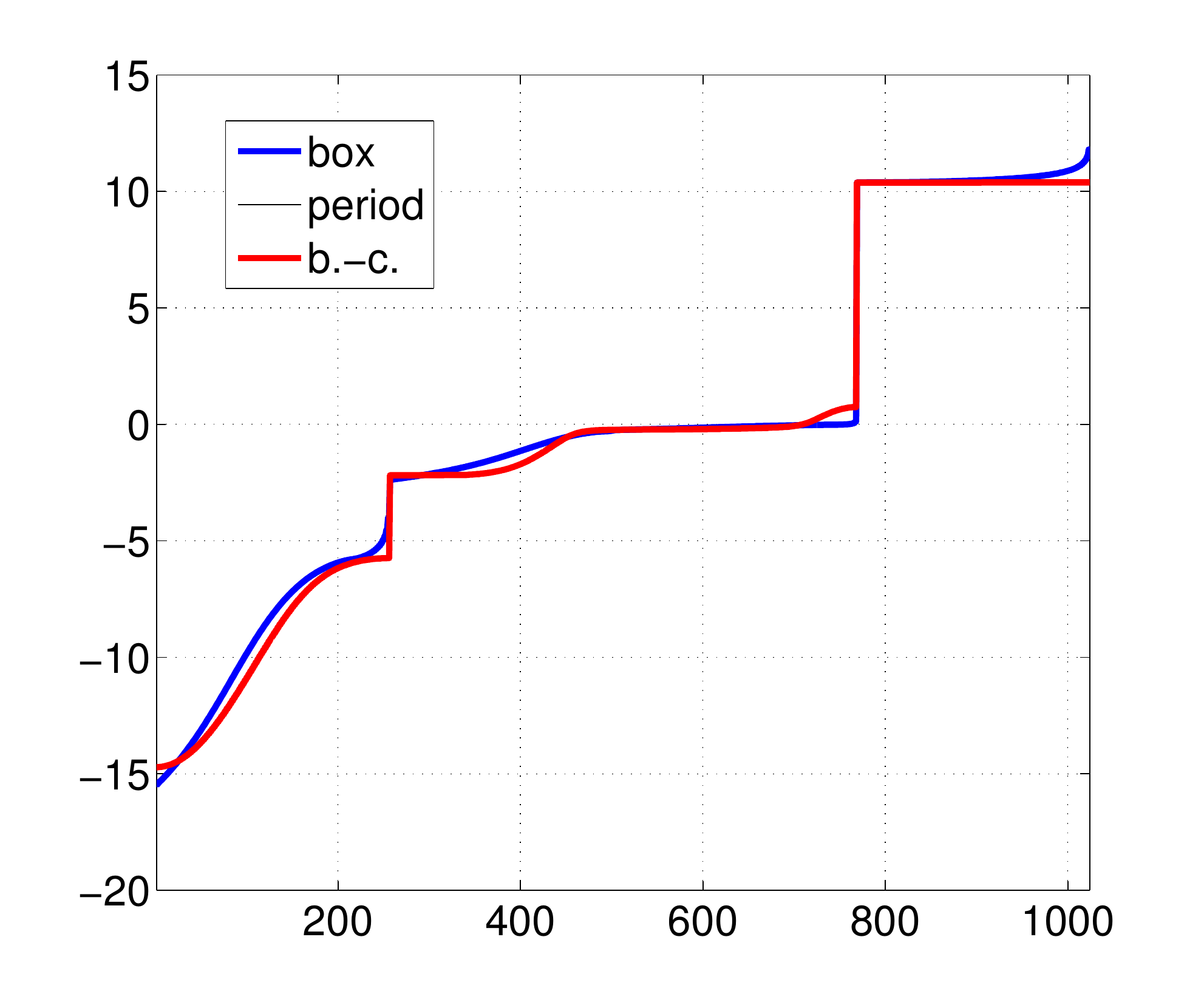}
\caption{\small  Spectrum of the core Hamiltonian in a box and in a periodic supercell
for $L=128,256$.}
\label{fig:3DCoreEig}  
\end{figure}

Figure \ref{fig:3DCoreEig} represents the spectrum of the core Hamiltonian 
in a box in comparison with that in a periodic supercell. We consider
different number of cells $L=128,256$, where $m_0=4$.  
The systematic difference between the eigenvalues in both cases can be observed
even at the limit of large $L$. This kind of spectral pollution effects have been discussed 
and theoretically analyzed in \cite{CancesDeLe:08}.
%
 
\begin{figure}[tbp]
\centering
\includegraphics[width=7.0cm]{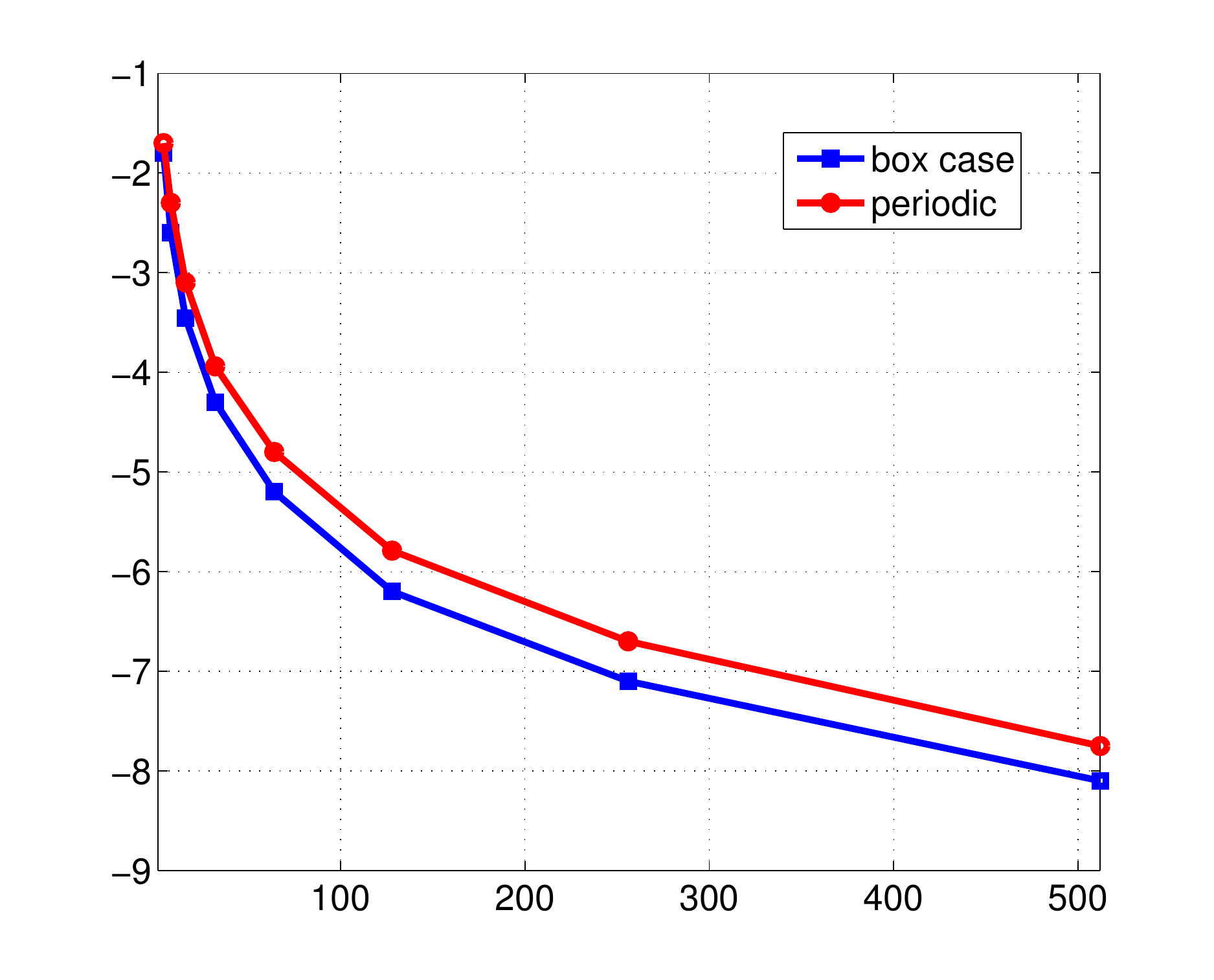} 
\caption{\small  Average energy per unit cell vs. $L$ for a $L\times 1\times 1$ lattice in a
3D rectangular ``tube``.}
\label{fig:3DAveEig}  
\end{figure}
Figure \ref{fig:3DAveEig} demonstrates the relaxation (for increasing $L$) 
of the average energy per unit cell
with $m_0=4$, for a $L\times 1\times 1$ lattice structure in a 3D rectangular "tube"   up
to $L=512$, for both periodic and open boundary conditions.

\section{Conclusions}\label{sec:Conclusions}

We introduced and analyzed the grid-based tensor approach to discretization 
and solution of the linearized Hartree-Fock equation in {\it ab initio} modeling 
of the lattice-structured molecular systems.
We describe methods and algorithms for banded (finite box) and block-circulant (periodic setting)
structured representation of the Fock matrix in GTO basis set (for the core Hamiltonian) 
and provide the numerical illustrations for both cases by implementing the algorithms in Matlab.


The sparse block structured representation to the Fock matrix combined with 
tensor techniques manifest several benefits: 
(a) the entries of the banded block structured Fock matrix are computed 
by 1D operations using low-rank CP tensors; 
(b) the storage size in the case of $L \times L \times L$ lattice is reduced to $O(L^d)\ll L^{2d}$;
(c) the $3$-level block-circulant Fock matrix in the periodic setting 
admits the low-rank tensor structure in the coefficients tensor, thus 
reducing the matrix diagonalization via conventional 3D FFT to the product of 1D Fourier transforms.

The main contributions include:

\begin{itemize}
 \item Fast computation of the Fock matrix by 1D matrix-vector operations 
by using low-rank tensors associated with a 3D spacial grid.
\item Analysis and numerical implementation of the multilevel 
banded/Toeplitz structure in the Fock matrix for the system in a box,
and the block circulant structure in periodic setting.   
\item Establishing the low-rank tensor structure in 
the diagonal blocks of the Fock matrix represented in the Fourier space, that
allows to reduce the storage size and diagonalization costs to $O(m_0^3 \,d\, L \log L)$.
\item Numerical tests illustrating the computational efficiency of the tensor-structured 
methods applied to the linearized Hartree-Fock equation for finite lattices and 
in periodic suprcell.
Numerical experiments verify the theoretical results on the 
asymptotic complexity estimates of the proposed algorithms.
\end{itemize}


The rigorous numerical study of the nonlinear  reduced 
Hartree-Fock eigenvalue problem for periodic and lattice-structured systems in a box
is a subject of the future research.

\section{Appendix} \label{sec_Append:ML-block-circToepl}

\subsection{Overview on block circulant matrices}
\label{ssec_Append:block-circ}

We recall that a one-level block circulant matrix  $A\in {\cal BC} (L,m_0)$ 
is defined by \cite{Davis},
\begin{equation}\label{eqn:block_c}
A=\operatorname{bcirc}\{A_0,A_1,...,A_{L-1}\}=
 \begin{bmatrix}
A_0     &  A_{L-1} &  \cdots  & A_{2} &  A_{1} \\
A_{1} &  A_0 &  \cdots  & \vdots & A_{2} \\
\vdots  &  \vdots & \ddots &  A_0 & \vdots  \\
A_{L-1}   &  A_{L-2}  & \cdots & A_{1} & A_0 \\
\end{bmatrix}
\in \mathbb{R}^{L m_0\times L m_0},
\end{equation}
where $A_k \in \mathbb{R}^{m_0\times m_0}$ for $k=0,1, \ldots ,L-1$, are matrices of general structure. 
The equivalent Kronecker product representation is defined by the associated matrix polynomial,
\begin{equation}\label{eqn:bcircPol}
 A= \sum\limits^{L -1}_{k=0} \pi^k \otimes A_k =:p_A(\pi),
\end{equation}
where $\pi=\pi_L\in \mathbb{R}^{L \times L}$ is the periodic downward shift 
(cycling permutation) matrix,
\begin{equation}\label{eqn:perShift}
\pi_L:=
 \begin{bmatrix}
0  &  0    &  \cdots &  0  &  1 \\ 
1  &  0   &  \cdots &  0  &  0 \\
\vdots  &\vdots & \ddots & \vdots & \vdots \\
0  &    \cdots  &  1 &  0  &  0 \\
0  &   0  &  \cdots &  1  &  0 \\
\end{bmatrix}
,
\end{equation}
and $\otimes$ denotes the Kronecker product of matrices. 

In the case $m_0=1$ a matrix $A\in {\cal BC} (L,1)$ defines a circulant matrix 
generated by its first column vector
$\widehat{a}=(a_0,...,a_{L-1})^T$. The associated scalar polynomial then reads
\[
 p_A(z):= a_0 + a_1 z + ... +a_{L-1} z^{L-1},
\]
so that (\ref{eqn:bcircPol}) simplifies to
\[
 A=p_A(\pi_L). 
\]
Let $\omega= \omega_L= \exp(-\frac{2\pi i}{L})$,  we denote by 
$$
F_L=\{f_{k\ell}\}\in \mathbb{R}^{L\times L}, \quad \mbox{with} \quad  
f_{k\ell}=\frac{1}{\sqrt{L}}\omega_L^{(k-1)(\ell-1)},\quad
k,l=1,...,L,
$$
the unitary matrix of Fourier transform. Since the shift matrix $\pi_L$ is diagonalizable 
in the Fourier basis, 
\begin{equation} \label{eqn:diagshift}
 \pi_L=F_L^\ast D_L F_L,\quad D_L= \mbox{diag}\{1,\omega,...,\omega^{L-1} \},
\end{equation}
the same holds for any circulant matrix,
\begin{equation} \label{eqn:circDiag}
 A = p_A(\pi_L) = F_L^\ast p_A(D_L) F_L,  
\end{equation}
where
\[
 p_A(D_L)=\mbox{diag}\{p_A(1),p_A(\omega),...,p_A(\omega^{L-1})\}= \mbox{diag}\{F_L a\}.
\]

Conventionally, we denote by $\mbox{diag}\{x\}$ a diagonal matrix generated by a vector $x$. 
Let $X$ be an $L  m_0\times m_0$ matrix obtained by concatenation of $m_0\times m_0$ 
matrices $X_k$, $k=0,...,L-1$, 
$X=\operatorname{conc}(X_0,...,X_{L-1})=[X_0,...,X_{L-1}]^T$. 
For example, the first block column in (\ref{eqn:block_c})
has the form $\operatorname{conc}(A_0,...,A_{L-1})$.
We denote by $\mbox{bdiag}\{X\}$ the $L m_0\times L m_0$ block-diagonal matrix of 
block size $L$ generated by $m_0\times m_0$ blocks $X_k$. 

It is known that similarly to the case of circulant matrices (\ref{eqn:circDiag}), 
block circulant matrix in ${\cal BC} (L,m_0)$ is unitary equivalent to the block diagonal one 
by means of Fourier transform via  representation (\ref{eqn:bcircPol}), see \cite{Davis}. 
In the following, we describe the block-diagonal representation of a matrix 
$A\in {\cal BC} (L,m_0)$ in the form that is convenient for generalization to the multi-level
block circulant matrices as well as for the description of FFT based implementation schemes. 
To that end, let us introduce the reshaping (folding) transform ${\cal T}_L$ that maps a 
$L m_0\times m_0$ matrix $X$ (i.e., the first block column in $A$) 
to $L\times m_0\times m_0$ tensor $B={\cal T}_L X$ by plugging 
the $i$th $m_0\times m_0$ block in $X$ into a slice $B(i,:,:)$. The respective unfolding returns 
the initial matrix $X={\cal T}_L' B$.
We denote by $\widehat{A}\in \mathbb{R}^{L m_0\times m_0}$ the first 
block column of a matrix $A\in {\cal BC} (L,m_0)$, with a shorthand notation 
$$
\widehat{A}=[A_0,A_1,...,A_{L-1}]^T,
$$ 
so that the $L\times m_0\times m_0$ tensor ${\cal T}_L \widehat{A}$ represents slice-wise
all generating $m_0\times m_0$ matrix blocks.  

\begin{proposition}\label{prop:eig_bcmatr}
For $A\in {\cal BC} (L,m_0)$ we have
\begin{equation} \label{eqn:bcircDiag}
 A= (F_L^\ast \otimes I_{m_0}) \operatorname{bdiag} 
\{ \bar{A}_0, \bar{A}_1,\ldots , \bar{A}_{L-1}\}
(F_L \otimes I_{m_0}),
\end{equation}
where 
\[
 \bar{A}_j = \sum\limits^{L -1}_{k=0} \omega_L^{jk} A_k \in \mathbb{C}^{m_0 \times m_0},
\]
can be recognized as the $j$-th $m_0\times m_0$ matrix block in block column 
${\cal T}_L'(F_L ({\cal T}_L \widehat{A}))$, such that 
\[
  \left[ \bar{A}_0, \bar{A}_1,\ldots , \bar{A}_{L-1}\right]^T = 
 {\cal T}_L'(F_L ({\cal T}_L \widehat{A})).
\]
A set of eigenvalues $\lambda$ of $A$ is then given by
\begin{equation}\label{eqn:lambdAbc}
\{\lambda | Ax = \lambda x, \; x\in \mathbb{C}^{L m_0}\}=
\bigcup\limits_{j=0}^{L-1} \{ \lambda |\bar{A}_j u = \lambda u, \; u \in \mathbb{C}^{m_0}  \}.
 \end{equation}
The eigenvectors corresponding to the spectral sets 
$$
\Sigma_j= \{\lambda_{j, m} |\bar{A}_j u_{j,m} = \lambda_{j,m} u_{j,m}, 
\; u_{j,m} \in \mathbb{C}^{m_0}\}, 
\quad j= 0,1,\ldots , L-1,\quad m=1,...,m_0,
$$ 
can be represented in the form
\begin{equation}\label{eqn:eigvecA}
 U_{j,m}=(F_L^\ast \otimes I_{m}) \bar{U}_{j,m},\quad \mbox{where} \quad 
\bar{U}_{j,m}= E_{[j]} \operatorname{vec}\, [u_{0,m},u_{1,m},...,u_{L-1,m}],
\end{equation}
with $E_{[j]}=\operatorname{diag}\{e_j\}\otimes I_{m_0} \in \mathbb{R}^{L m_0\times L m_0} $, and
$e_j\in \mathbb{R}^{L}$ being the $j$th Euclidean basis vector.
\end{proposition}
\begin{proof}
We combine representations (\ref{eqn:bcircPol}) and (\ref{eqn:diagshift}) to obtain
\begin{align}\label{eqn:Bcircdiag}
A & = \sum\limits^{L -1}_{k=0} \pi^k \otimes A_k = 
      \sum\limits^{L -1}_{k=0} (F_L^\ast D^k F_L) \otimes A_k \\ \nonumber
  & = (F_L^\ast \otimes I_{m_0}) (\sum\limits^{L -1}_{k=0}  D^k  
       \otimes A_k)(F_L \otimes I_{m_0}) \\ \nonumber
  & = (F_n^\ast \otimes I_m)(\sum\limits^{L -1}_{k=0} 
      \mbox{bdiag}\{A_k,\omega_L^k A_k,...,\omega_L^{k(L-1)}A_{k} \}  )  
       (F_L \otimes I_{m_0})\\ \nonumber
  & = (F_L^\ast \otimes I_{m_0})  
      \mbox{bdiag}\{\sum\limits^{L -1}_{k=0} A_k,\sum\limits^{L -1}_{k=0} \omega_L^k A_k,...,
      \sum\limits^{L -1}_{k=0} \omega_L^{k(L-1)}A_{k} \}    (F_L \otimes I_{m_0})\\ \nonumber
  & = (F_L^\ast \otimes I_{m_0}) \mbox{bdiag}_{m_0 \times m_0} 
  \{{\cal T}_L'(F_L ({\cal T}_L \widehat{A}))\} (F_L \otimes I_{m_0}),  \nonumber
\end{align}
where the final step follows by the definition of FT matrix and by the construction of 
${\cal T}_L$. 
The structure of eigenvalues and eigenfunctions  then follows by straightforward 
calculations with block-diagonal matrices.
\end{proof}

The next statement describes  the block-diagonal form for a class of symmetric BC 
matrices, ${\cal BC}_s (L,m_0)$, 
that is a simple corollary of \cite{Davis}, Proposition \ref{prop:eig_bcmatr}. 
In this case we have $A_0=A_0^T$, and $A_k^T=A_{L-k}$, $k=1,...,L-1$.
\begin{corollary}\label{cor:eig_symbcmatr}
 Let $A\in {\cal BC}_s (L,m_0)$  be symmetric, then $A$ is unitary similar
to a Hermitian block-diagonal matrix, i.e., $A$ is of the form 
\begin{equation}\label{eqn:F_bc}
 A= (F_L \otimes I_{m_0}) \operatorname{bdiag} (\tilde{A}_0, \tilde{A}_1,\ldots , \tilde{A}_{L-1})
(F_L^\ast \otimes I_{m_0}),
\end{equation}
where $I_{m_0}$ is the $m_0\times m_0$ identity matrix.
The matrices $\tilde{A}_j \in \mathbb{C}^{m_0\times m_0}$, $j= 0,1,\ldots , L-1$, are 
defined for even $n\geq 2$ as
\begin{equation}\label{eqn:symBc}
 \tilde{A}_j =A_0 + \sum\limits^{L/2-1}_{k=1} (\omega^{kj}_L A_k + 
\widehat{\omega}^{kj}_L A^T_k) + (-1)^j A_{L/2}.
\end{equation}
\end{corollary}

Corollary \ref{cor:eig_symbcmatr} combined with Proposition \ref{prop:eig_bcmatr} describes
a simplified  structure of spectral data in the symmetric case.
Notice that the above representation imposes the symmetry of each 
real-valued diagonal blocks $\tilde{A}_j \in \mathbb{R}^{m_0\times m_0}, \; j= 0,1,\ldots , L-1$,
in (\ref{eqn:F_bc}).

\subsection{Multilevel block circulant/Toeplitz matrices}
\label{ssec_Append:ML_block-circ}

Furthermore,  we describe the extension of (one-level) block circulant matrices to multilevel structure.
First, we recall the main notions of multilevel block circulant (MBC) matrices 
with the particular focus on the three-level case. 
Given a multi-index ${\bf L}=(L_1, L_2, L_3)$,
we denote $|{\bf L}|=L_1\, L_2\, L_3$. 
A matrix class ${\cal BC} (d,{\bf L},m_0)$ ($d=1,2,3$)
of $d$-level block circulant matrices can be introduced by the following recursion.
\begin{definition}\label{def:Bcirc}
For $d=1$, define a class of one-level block circulant matrices by
${\cal BC} (1,{\bf L},m)\equiv {\cal BC} (L_1,m)$ (see \S\ref{ssec_Append:block-circ}), 
where ${\bf L}=(L_1,1,1)$. For $d=2$, we say that a matrix 
$A\in \mathbb{R}^{|{\bf L}|m_0 \times |{\bf L}| m_0}$ belongs to a class
${\cal BC} (d,{\bf L},m_0)$ if
\[
 A = \operatorname{bcirc}(A_1,...,A_{L_1})\quad \mbox{with}\quad 
 A_j\in {\cal BC}(d-1,{\bf L}_{[1]},m_0),\; j=1,...,L_1,
\]
where ${\bf L}_{[1]}=(L_2,L_3)\in \mathbb{N}^{d-1} $. Similar recursion applies to the case $d=3$.
\end{definition}

Likewise to the case of one-level BC matrices, it can be seen that a matrix 
$A \in {\cal BC} (d,{\bf L},m_0)$, $d=1,2,3$, 
of size $|{\bf L}| m_0 \times  |{\bf L}| m_0$ is completely defined (parametrized) by a $d$th order 
matrix-valued tensor 
${\bf A}=[A_{k_1 ... k_d}]$ of size $L_1\times ... \times L_d $,
($k_\ell=1,...,L_\ell$, $\ell=1,...,d$), with $m_0\times m_0$ matrix entries $A_{k_1 ... k_d}$,  
obtained by folding of the generating first column vector in $A$.

Recall that a symmetric block Toeplitz matrix  $A\in {\cal BT}_s (L,m_0)$ 
is defined by \cite{Davis},
\begin{equation}\label{eqn:block_SToepl}
A=\operatorname{BToepl}_s\{A_0,A_1,...,A_{L-1}\}=
 \begin{bmatrix}
A_0     &  A_{1}^T &  \cdots  & A_{L-2}^T &  A_{L-1}^T \\
A_{1} &  A_0 &  \cdots  & \vdots & A_{L-2}^T \\
\vdots  &  \vdots & \ddots &  A_0 & \vdots  \\
A_{L-1}   &  A_{L-2}  & \cdots & A_{1} & A_0 \\
\end{bmatrix}
\in \mathbb{R}^{L m_0\times L m_0},
\end{equation}
where $A_k \in \mathbb{R}^{m_0\times m_0}$ for $k=0,1, \ldots ,L-1$, 
is a matrix of a general structure.

Similar to Definition \ref{def:Bcirc}, a matrix class ${\cal BT}_s (d,{\bf L},m_0)$ 
of symmetric $d$-level block Toeplitz matrices can be introduced by the following recursion.
\begin{definition}\label{def:BToepl}
For $d=1$, ${\cal BT}_s (1,{\bf L},m_0)\equiv {\cal BT}_s (L_1,m_0)$ is the class 
of one-level symmetric block circulant matrices with ${\bf L}=(L_1,1,1)$.
For $d=2$ we say that a matrix 
$A\in \mathbb{R}^{|{\bf L}|m \times |{\bf L}| m_0}$ belongs to a class
${\cal BT}_s (d,{\bf L},m_0)$ if
\[
 A= \operatorname{btoepl}_s(A_1,...,A_{L_1})\quad 
 \mbox{with}\quad A_j\in {\cal BT}_s(d-1,{\bf L_{[1]}},m_0),\; j=1,...,L_1.
\]
Similar recursion applies to the case $d=3$.
\end{definition} 

\subsection{ Rank-structured tensor formats}
\label{ssec_Append:Tensors}

We consider a tensor of order $d$, as a multidimensional array numbered 
by a $d$-tuple index set, ${\bf A}=[a_{i_1,...,i_d}]\; \in 
\mathbb{R}^{n_1 \times \ldots \times n_d}$. 
A tensor is an element of a linear vector space equipped with the Euclidean scalar product.
In particular, a tensor with equal sizes $n_\ell =n$, $\ell=1,\ldots d$, is 
denoted as $n^{\otimes d}$ tensor. 
The  required storage for entry-wise representation of tensors
scales exponentially in the dimension, $n^d$, (the so-called ''curse of dimensionality``). 
To get rid of exponential scaling in the dimension, one can apply  the 
rank-structured separable representations of multidimensional tensors.

The rank-$1$ canonical tensor, $ {\bf A} = {\bf u}^{(1)}\otimes ... 
\otimes {\bf u}^{(d)}\in \mathbb{R}^{n_1 \times \ldots \times n_d}$, with entries 
$a_{i_1,\ldots i_d}=  a^{(1)}_{i_1}\cdot \cdot\cdot a^{(d)}_{i_d} $ 
requires only $d n$ numbers to store it. 
A tensor in the $R$-term canonical format (CP tensors) is defined by the parametrization
 \begin{equation}
\label{CP_form}
   {\bf A} = \sum\limits_{k =1}^{R} c_k
   {\bf u}_k^{(1)}  \otimes \ldots \otimes {\bf u}_k^{(d)},
 \quad  c_k \in \mathbb{R},
\end{equation}
where $u_k^{(\ell)}$ are normalized vectors, and $R$ is called the canonical rank of a tensor.
The storage size is bounded by $d n R \ll n^d$.

Given the rank parameter ${\bf r}=(r_1,...,r_d)$,
a tensor in the rank-$\bf r$ Tucker format is defined by the parametrization
\[
  {\bf A} ={\sum}_{\nu_1 =1}^{r_1}\ldots
{\sum}^{r_d}_{{\nu_d}=1} \beta_{\nu_1, \ldots ,\nu_d}
\,  {\bf v}^{(1)}_{\nu_1} \otimes \ldots \otimes {\bf v}^{(d)}_{\nu_d},\quad \ell=1,\ldots,d, 
\]
completely specified by a set of orthonormal vectors 
${\bf v}^{(\ell)}_{\nu_\ell}\in \mathbb{R}^{n_\ell}$,
and the Tucker core tensor $\boldsymbol{\beta}=[\beta_{\nu_1,...,\nu_d}]$.
The storage demand is bounded by $|{\bf r}| + (r_1+\ldots +r_d)n$.

The remarkable approximating properties of the Tucker and canonical tensor decomposition 
applied to the wide class of function related tensors  were revealed in 
\cite{Khor1:06,HaKhtens:04I,KhKh3:08}, 
promoting using  tensor tools for the numerical treatment of the multidimensional PDEs.
It was proved
for the CP/Tucker decomposition of some classes of function related tensors 
that the rank-$\bf r$ Tucker approximation with $ r=O(\log n)$, i.e., $r_\ell \ll n $, 
provides the exponentially small error of order $e^{-\alpha  r}$, \cite{Khor1:06}
(here we discuss for simplicity the equal rank decompositions, $r_\ell=r$).

In the case of many spacial dimensions the product type tensor formats provide 
the stable rank-structured approximation.
The matrix-product states (MPS) decomposition was long since used 
in quantum chemistry and quantum information theory, see the survey paper \cite{Scholl:11}.
The particular case of MPS representation is called a tensor train (TT) format 
\cite{OselTy:09}.
The quantics-TT (QTT) tensor approximation method 
for functional $n$-vectors  was introduced in \cite{KhQuant:09} and shown to 
provide the logarithmic complexity, $O(d\log n)$, on the wide class of generating functions.
Furthermore, a combination of different tensor formats proved to be successful in 
the numerical solution of the multidimensional PDEs \cite{KhorSurv:10,DoKh:13}.

\begin{footnotesize}

\end{footnotesize}

\end{document}